\documentclass[12pt]{amsart}
\usepackage{ amsmath, amsthm, amsfonts, amssymb, color}
 \usepackage{mathrsfs}
\usepackage{amsfonts, amsmath}
 \usepackage{amsmath,amstext,amsthm,amssymb,amsxtra}
 \usepackage{txfonts}
 \usepackage[colorlinks, citecolor=blue,pagebackref,hypertexnames=false]{hyperref}
 \allowdisplaybreaks
 \usepackage{pgf,tikz}
 \usepackage{txfonts}
\usepackage{cite}

\DeclareMathOperator*{\esssup}{ess\,sup}

\begin{document}

\hfuzz=6pt

\widowpenalty=10000

\newtheorem{cl}{Claim}
\newtheorem{theorem}{Theorem}[section]
\newtheorem{proposition}[theorem]{Proposition}
\newtheorem{coro}[theorem]{Corollary}
\newtheorem{lemma}[theorem]{Lemma}
\newtheorem{definition}[theorem]{Definition}

\newtheorem{assum}{Assumption}[section]
\newtheorem{example}[theorem]{Example}
\newtheorem{remark}[theorem]{Remark}
\renewcommand{\theequation}
{\thesection.\arabic{equation}}

\def\SL{\sqrt H}

\newcommand{\mar}[1]{{\marginpar{\sffamily{\scriptsize
        #1}}}}

\newcommand{\as}[1]{{\mar{AS:#1}}}

\newcommand\R{\mathbb{R}}
\newcommand\RR{\mathbb{R}}
\newcommand\CC{\mathbb{C}}
\newcommand\NN{\mathbb{N}}
\newcommand\ZZ{\mathbb{Z}}
\def\RN {\mathbb{R}^n}
\renewcommand\Re{\operatorname{Re}}
\renewcommand\Im{\operatorname{Im}}

\newcommand{\mc}{\mathcal}
\newcommand\D{\mathcal{D}}
\def\hs{\hspace{0.33cm}}
\newcommand{\la}{\lambda}
\def \l {\lambda}
\newcommand{\eps}{\varepsilon}
\newcommand{\pl}{\partial}
\newcommand{\supp}{{\rm supp}{\hspace{.05cm}}}
\newcommand{\x}{\times}
\newcommand{\lag}{\langle}
\newcommand{\rag}{\rangle}

\newcommand\wrt{\,{\rm d}}

\newcommand {\BMO}{{\mathrm{BMO}}}
\newcommand {\Rn}{{\mathbb{R}^{n}}}
\newcommand {\rb}{\rangle}
\newcommand {\lb}{{\langle}}
\newcommand {\HT}{\mathcal{H}}
\newcommand {\Hp}{\mathcal{H}^{p}_{FIO}(\Rn)}
\newcommand {\ud}{\mathrm{d}}
\newcommand {\Sp}{S^{*}\Rn}
\newcommand {\Sw}{\mathcal{S}}
\newcommand {\w}{{\omega}}
\newcommand {\ph}{{\varphi}}
\newcommand {\para}{{\mathrm{par}}}
\newcommand {\N}{{{\mathbb N}}}
\newcommand {\Z}{{{\mathbb Z}}}
\newcommand {\F}{{\mathcal{F}}}
\newcommand {\C}{{\mathbb C}}
\newcommand {\vanish}[1]{\relax}
\newcommand {\ind}{{\mathbf{1}}}

\newcommand{\wt}{\widetilde}

\newtheorem{corollary}[theorem]{Corollary}

\title[Hardy spaces for Fourier Integral Operators]
 {Characterizations of the Hardy space $\mathcal{H}_{FIO}^{1}(\mathbb{R}^{n})$ for Fourier Integral Operators }

\makeatletter
\@namedef{subjclassname@2020}{\textup{2020} Mathematics Subject Classification}
\makeatother


\author{Zhijie Fan, \ Naijia Liu, \ Jan Rozendaal \ and  Liang Song}
\address{
   Zhijie Fan, School of Mathematics and Statistics, Wuhan University, Wuhan 430072, China}
\email{ZhijieFan@whu.edu.cn}

\address{
   Naijia Liu,
    Department of Mathematics,
    Sun Yat-sen University,
    Guangzhou, 510275,
    P.R.~China}
\email{liunj@mail2.sysu.edu.cn}

\address{
Jan Rozendaal,
Institute of Mathematics, Polish Academy of Sciences,
ul.~\'{S}niadeckich 8,
00-656 Warsaw,
Poland}
\email{jrozendaal@impan.pl}

\address{
   Liang Song,
    Department of Mathematics,
    Sun Yat-sen University,
    Guangzhou, 510275,
    P.R.~China}
\email{ songl@mail.sysu.edu.cn}

\begin{abstract}
The Hardy spaces for Fourier integral operators $\mathcal{H}_{FIO}^{p}(\mathbb{R}^{n})$, for $1\leq p\leq \infty$, were introduced by Smith in \cite{Smith98a} and Hassell et al.~in \cite{HaPoRo20}. In this article, we give several equivalent characterizations of $\mathcal{H}_{FIO}^{1}(\mathbb{R}^{n})$, for example in terms of Littlewood--Paley $g$ functions and maximal functions. This answers a question from \cite{Rozendaal21}. We also give several applications of the characterizations.
\end{abstract}

\maketitle

\section{Introduction}

\subsection{Overview}
Hardy spaces have long been of great importance in harmonic analysis and related fields. For example, the classical Hardy space $H^{1}(\Rn)$ is the natural harmonic analytic substitute of 
$L^{1}(\Rn)$ for the study of
singular integral operators (see \cite{Stein93,Grafakos14b}). And in recent years, the theory of adapted Hardy spaces has played a major role in the analysis of parabolic and elliptic partial differential equations with rough coefficients. These adapted Hardy spaces are the natural substitutes of $L^{p}(\Rn)$ when the equation under consideration is not well behaved on $L^{p}(\Rn)$ for certain $1\leq p\leq \infty$ (see \cite{Auscher07, HoLuMiMiYa11}). In turn, there are many characterizations of Hardy spaces, for example in terms of
area functionals, Littlewood--Paley g functions and maximal functions. These characterizations are powerful harmonic analytic tools, as they allow for different methods of tackling a given problem.

Although singular integral operators are bounded on $H^{1}(\RN)$, and thus bounded from $H^{1}(\Rn)$ to $L^{1}(\Rn)$, the situation is quite different for oscillatory integral operators. Indeed, Fourier integral operators (FIOs) of order zero are in general not bounded from $H^{1}(\RN)$ to $L^{1}(\Rn)$ unless $n=1$. FIOs are typical examples of oscillatory integrals, and they arise naturally in classical analysis and partial differential equations, for example as the solution operators to wave equations with smooth coefficients (see \cite{Hormander71,Duistermaat11,Sogge17}). As shown by Seeger, Sogge and Stein in \cite{SeSoSt91}, a compactly supported FIO $T$ of order zero, associated with a local canonical graph, satisfies $T: \lb D\rb^{-\frac{n-1}{2}}H^{1}(\RN) \to L^{1}(\RN)$, and the exponent $\frac{n-1}{2}$ cannot be improved. Here $\lb D\rb^{-\frac{n-1}{2}}$ is the Fourier multiplier with symbol $\lb \xi\rb^{-\frac{n-1}{2}}=(1+|\xi|^{2})^{-\frac{n-1}{4}}$. This result is often summarized by saying that FIOs ``lose" $(n-1)/2$ derivatives on $H^{1}(\RN)$. Using interpolation, the $L^{2}$-boundedness of FIOs, and duality, one in turn obtains optimal results about the $L^{p}$-boundedness of FIOs, and thereby also the optimal $L^{p}$-regularity for wave equations with smooth coefficients.

Although the loss of derivatives for FIOs on $L^{p}(\R^{n})$ is unavoidable for $p\neq 2$ and $n>1$, one could argue that $L^{p}$-spaces are, in some ways, not the right function spaces for the analysis of FIOs. Indeed, in \cite{Smith98a} Smith introduced a Hardy space, denoted by $\mathcal{H}_{FIO}^{1}(\mathbb{R}^{n})$, which is invariant under suitable FIOs of order zero, and this space is large enough to recover the results in \cite{SeSoSt91}. Recently, in \cite{HaPoRo20}, Smith's work was extended to a full scale of Hardy spaces $(\mathcal{H}_{FIO}^{p}(\mathbb{R}^{n}))_{1\leq p\leq\infty}$. 
These spaces are invariant under FIOs of order zero, and they satisfy Sobolev embeddings which allow one to directly recover the optimal results about $L^p$-boundedness of FIOs. 

Apart from the intrinsic interest in determining the natural function spaces for FIOs, the Hardy spaces for FIOs were introduced with applications to wave equations with rough coefficients and nonlinear wave equations in mind. Indeed, a common method of solving rough or nonlinear equations is to use iterative constructions to build a solution. However, such a process breaks down if one loses derivatives in each iteration step. 
On the other hand, since $\Hp$ is invariant under FIOs, 
one can 
iterate 
to build a solution on $\Hp$, and then afterwards use the Sobolev embeddings for $\Hp$ to deduce sharp results on $L^{p}(\Rn)$. 
This approach was 
used in 
\cite{Hassell-Rozendaal20} to obtain the optimal fixed-time $L^{p}$-regularity for wave equations with rough coefficients, for $1<p<\infty$. This constitutes the first extension of the seminal work in \cite{SeSoSt91} to a general class of wave equations with rough coefficients. Function spaces related to the Hardy spaces for FIOs were utilized in a different manner in \cite{Frey-Portal20}, to extend 
\cite{SeSoSt91} to a specific class of rough wave equations. We also note that the Hardy spaces for FIOs were applied in \cite{Rozendaal22b} to obtain improved local smoothing estimates for the Euclidean wave equation, by connecting these spaces to the $\ell^{p}$-decoupling inequality from \cite{Bourgain-Demeter15}.

In \cite{HaPoRo20} (and implicitly already in \cite{Smith98a}), the definition of $\Hp$ follows a template from the theory of adapted Hardy spaces, using embeddings into tent spaces, where much of the subsequent analysis can be performed. However, an intrinsic difference between the theory of parabolic and elliptic equations, for which the theory  of Hardy spaces has been widely used, and that of hyperbolic equations, is that the latter exhibit propagation of singularities. This is the phenomenon whereby singularities of the initial data are moved around by the solution operators, and it takes place on phase space, i.e.~on the cotangent bundle $T^{*}\R^{n}=\R^{n}\times\Rn$ of $\Rn$. Hence, to obtain function spaces that are invariant under FIOs, one needs to move from $\Rn$ to phase space. 
This is achieved 
using wave packet transforms, to embed function spaces over $\Rn$ into tent spaces over the cosphere bundle $\Sp=\Rn\times S^{n-1}$ of $\Rn$.  One can then apply the established theory of tent spaces (see \cite{CoMeSt85, Amenta14}).

Although this definition of $\Hp$ leads to a robust theory which builds on tools from 
other parts of harmonic analysis, it has several drawbacks. For example, the resulting function space norm involves a relatively technical conical square function over $\Sp$, 
and it is natural to wonder whether there are descriptions of $\Hp$ that are easier to work with. Moreover, $H^{1}(\Rn)$ can be characterized in a variety of ways, and one might ask whether similar characterizations hold for the Hardy spaces for FIOs. Finally, although the definition of $\Hp$ in \cite{HaPoRo20,Smith98a} comes with kernel conditions which guarantee boundedness on $\Hp$, prior work on rough wave equations on $L^{2}(\Rn)$ (see \cite{Smith98b,Smith06,Smith14,Smith-Tataru05,Tataru00,Tataru01,Tataru02})  makes crucial use of techniques from Littlewood--Paley theory and paradifferential calculus, and some of these tools are not available in the theory of tent spaces.

To address these issues, in \cite{Rozendaal21} 
several characterizations of $\mathcal{H}_{FIO}^{p}(\mathbb{R}^{n})$ were proved for $1<p<\infty$. First, $\mathcal{H}_{FIO}^{p}(\mathbb{R}^{n})$ is characterized in terms of $L^p(\mathbb{R}^{n})$-norms of parabolic frequency localizations. 
As a corollary, any characterization of $L^p(\mathbb{R}^{n})$ yields a corresponding version for $\mathcal{H}_{FIO}^{p}(\mathbb{R}^{n})$. In particular, 
one obtains characterizations of $\Hp$ in terms of Littlewood--Paley $g$ functions and 
maximal functions. These characterizations are more amenable to direct calculations, and they allow one to incorporate 
tools from Littlewood--Paley theory and paradifferential calculus. Such tools were subsequently used in \cite{Rozendaal20,Rozendaal22} to obtain mapping properties of rough pseudodifferential operators on $\Hp$ for $1<p<\infty$, and these mapping properties in turn play a crucial role in the proof of the optimal $L^{p}$-regularity for wave equations with rough coefficients in \cite{Hassell-Rozendaal20}.

In fact, the restriction in \cite{Hassell-Rozendaal20} to $1<p<\infty$ is due to the restriction to such $p$ in the main results of \cite{Rozendaal20,Rozendaal22}. Since the proofs of those results rely 
on the equivalent characterizations of $\Hp$ from \cite{Rozendaal21}, it is 
relevant for applications to rough wave equations to extend the characterizations in \cite{Rozendaal21} to $p=1$. Unfortunately, the methods of \cite{Rozendaal21} do not apply directly for $p=1$ or $p=\infty$, and it was left as an open question whether similar characterizations hold for $\HT^{1}_{FIO}(\Rn)$ and $\HT^{\infty}_{FIO}(\Rn)$.

In the present article, we answer the question in \cite{Rozendaal21} regarding $\HT^{1}_{FIO}(\Rn)$, by obtaining several equivalent characterizations of $\HT^{1}_{FIO}(\Rn)$, for example in terms of Littlewood--Paley g functions and maximal functions. These characterizations are similar to those in \cite{Rozendaal21}, and they allow one to incorporate 
tools from other parts of harmonic analysis for the study of FIOs. In particular, this allows for 
possible extensions of the results for rough wave equations in \cite{Hassell-Rozendaal20} 
to $p=1$.

In this article we also give a more direct application of the equivalent characterizations, and we show how one can use them to calculate the $\HT^{1}_{FIO}(\Rn)$-norm of certain functions.

\subsection{Statement of results}

To make our results precise, we first recall the definition of $\mathcal{H}_{FIO}^{p}(\mathbb{R}^{n})$  for $1\leq p<\infty$. Throughout, fix $n\geq 2$. The results in this article go through for $n=1$, but they reduce to classical statements about the local Hardy space $\HT^{1}(\R)$.

Let $\Sp=\Rn\times S^{n-1}$ be the cosphere bundle over $\mathbb{R}^{n}$, endowed with the standard measure $\ud x\ud\omega$ and
with a metric $d$ which arises from contact geometry (see Section \ref{subsec:metric}). We note that $(\Sp, d, \ud x\ud\omega)$ is a doubling metric measure space. For $\sigma>0$ and $(x,\w)\in\Sp$, we let $B_{\sqrt{\sigma}}(x,\omega):=\big\{(y,\nu)\in \Sp\mid d(y,\nu;x,\omega)< \sqrt{\sigma}\big\}$ be the ball around $(x,\omega)$ of radius $\sqrt{\sigma}$ with respect to the metric $d$. Throughout, fix a $q\in C_{c}^{\infty}(\mathbb{R}^{n})$ such that $q(\xi)=1$ for $|\xi|\leq 2$, and let $q(D)$ be the corresponding Fourier multiplier operator. Also, for $0<\sigma<1$ we let $\theta_{\nu,\sigma}\in C^{\infty}_{c}(\Rn)$ be a smooth function localized to the high frequency region $\big{\{}\xi \in \mathbb{R}^{n}\mid   |\xi|\eqsim \sigma^{-1},\big|\frac{\xi}{|\xi|}-\nu\big|\eqsim \sigma^{\frac{1}{2}}\big\}$ (see \eqref{the def of theta} for the precise definition of $\theta_{\nu,\sigma}$). For $f\in\Sw'(\Rn)$ and $(x,\w)\in\Sp$, set
\begin{equation}\label{eq:conical}
S(f)(x,\omega):=\bigg(\int_{0}^{1}\fint_{B_{\sqrt{\sigma}}(x,\omega)}|\theta_{\nu,\sigma}(D)f(y)|^{2}\,\ud y\ud\nu\frac{\ud\sigma}{\sigma}\bigg)^{1/2}.
\end{equation}
We can now define $\Hp$ for $1\leq p<\infty$.

\begin{definition}\label{Area function}
For $p\in[1,\infty)$, let $\mathcal{H}_{FIO}^{p}(\mathbb{R}^{n})$ consist of all $f\in \mathcal{S}'(\RN)$ such that $S(f)\in L^{p}(\Sp)$ and $q(D)f\in L^{p}(\Rn)$, endowed with the norm
\[
\|f\|_{\mathcal{H}_{FIO}^{p}(\mathbb{R}^{n})}:=\bigg(\int_{\Sp}\big(S(f)(x,\omega)\big)^p \ \ud x\ud\omega\bigg)^{1/p}+\|q(D)f\|_{L^p(\RN)}.
\]
\end{definition}
We note that this is not the original definition of $\Hp$ from \cite{HaPoRo20}. However, it follows from \cite[Corollary 3.8]{Rozendaal21} that Definition \ref{Area function} is equivalent to the original definition.

To define $\HT^{\infty}_{FIO}(\Rn)$ one has to replace the conical square function in \eqref{eq:conical} by a Carleson measure condition (see \cite[Section 6]{HaPoRo20}). However, $\HT^{\infty}_{FIO}(\Rn)$ will not play a significant role in this article, and for our purposes it suffices to define $\HT^{\infty}_{FIO}(\Rn)$ as the dual of $\HT^{1}_{FIO}(\Rn)$ (see \cite[Proposition 6.8]{HaPoRo20}):
\begin{equation}\label{eq:Hinfty}
\HT^{\infty}_{FIO}(\Rn)=(\HT^{1}_{FIO}(\Rn))^{*}.
\end{equation}
Here the duality pairing is the standard duality pairing $\lb f,g\rb_{\Rn}$ for $f\in \HT^{\infty}_{FIO}(\Rn)\subseteq\Sw'(\Rn)$ and $g\in \Sw(\Rn)\subseteq\HT^{1}_{FIO}(\Rn)$.

Next, we define the  Littlewood-Paley g function for FIOs as follows: for $f\in\Sw'(\Rn)$ and $(x,\omega)\in \Sp$, set
\begin{equation}\label{eq:vertical}
G(f)(x,\omega):=\bigg(\int_{0}^{1}|\theta_{\omega,\sigma}(D)f(x)|^{2}\frac{\ud\sigma}{\sigma}\bigg)^{1/2}.
\end{equation}
We can then introduce the second function space of interest in this article.

\begin{definition}\label{def:HardyG}
Let $\mathcal{H}_{FIO,G}^{1}(\mathbb{R}^{n})$ consist of all $f\in \Sw'(\mathbb{R}^{n})$ such that $G(f)\in L^{1}(\Sp)$ and $q(D)f\in L^{1}(\mathbb{R}^{n})$, endowed with the norm
\[
\|f\|_{\mathcal{H}_{FIO,G}^{1}(\mathbb{R}^{n})}:=\|G(f)\|_{L^{1}(\Sp)}+\|q(D)f\|_{L^{1}(\mathbb{R}^{n})}.
\]
\end{definition}

By \cite[Proposition 2.1 and Remark 2.2]{AuHoMa12}, the following continuous inclusion holds: $\mathcal{H}_{FIO}^{1}(\mathbb{R}^{n})\subseteq \mathcal{H}_{FIO,G}^{1}(\mathbb{R}^{n})$. However, until now it was not clear whether one also has $\mathcal{H}_{FIO,G}^{1}(\mathbb{R}^{n})\subseteq \mathcal{H}_{FIO}^{1}(\mathbb{R}^{n})$. In this article we show that this inclusion also holds, so that $\mathcal{H}_{FIO}^{1}(\mathbb{R}^{n})=\mathcal{H}_{FIO,G}^{1}(\mathbb{R}^{n})$.

We will give two additional characterizations of $\HT^{1}_{FIO}(\Rn)$. To state these, let $\alpha>0$ and, for $f\in \mathcal{S^{\prime}}(\mathbb{R}^{n})$ and $(x,\omega)\in\Sp$, set
\[
\mathcal{G}_{\alpha}^{*}(f)(x,\omega):=\bigg(\int_{0}^{1}\int_{\Sp}
\frac{|\theta_{\nu,\sigma}(D)f(y)|^{2}}{\sigma^{n}(1+\sigma^{-1}
d(x,\omega;y,\nu)^{2})^{n\alpha}}\, \ud y\ud\nu\frac{\ud\sigma}{\sigma}\bigg)^{1/2}.
\]
Also, let $\Phi\in\Sw(\Rn)$ be a Schwartz function such that $\Phi(0)=1$, and for $\sigma>0$ and $\xi\in\Rn$ let $\Phi_{\sigma}(\xi):=\Phi(\sigma\xi)$. The function $\varphi_{\omega}\in C^{\infty}(\Rn)$ which occurs below is supported on a paraboloid in the direction of $\w\in S^{n-1}$, and it is defined in Section \ref{sub1}. We recall that a tempered distribution $f\in\Sw'(\Rn)$ is a \emph{bounded distribution} if $f\ast g\in L^{\infty}(\Rn)$ for all $g\in\Sw(\Rn)$.

\begin{definition}\label{def:Hardyothers}
Let $\mathcal{H}_{FIO,\max}^{1}(\mathbb{R}^{n})$ consist of all $f\in \Sw'(\mathbb{R}^{n})$ such that $\varphi_{\omega}(D)f$
 is a bounded distribution for almost all $\omega \in S^{n-1}$, $\int_{\Sp}\sup\limits_{\sigma>0}|\Phi_{\sigma}(D)\varphi_{\omega}(D)f(x)|\,\ud x\ud\omega <\infty$, and $q(D)f\in L^{1}(\mathbb{R}^{n})$, endowed with the norm
\[
\|f\|_{\mathcal{H}_{FIO,\max}^{1}(\mathbb{R}^{n})}:=\int_{\Sp}\sup\limits_{\sigma>0}|\Phi_{\sigma}(D)\varphi_{\omega}(D)f(x)|\,\ud x\ud\omega+\|q(D)f\|_{L^{1}(\mathbb{R}^{n})}.
\]
Let $\mathcal{H}_{FIO,\mathcal{G}_{\alpha}^{*}}^{1}(\mathbb{R}^{n})$ consist of all $f\in \Sw'(\mathbb{R}^{n})$ such that $\mathcal{G}_{\alpha}^{*}(f)\in L^{1}(\Sp)$ and $q(D)f\in L^{1}(\mathbb{R}^{n})$, endowed with the norm
\[
\|f\|_{\mathcal{H}_{FIO,\mathcal{G}_{\alpha}^{*}}^{1}(\mathbb{R}^{n})}:=\|\mathcal{G}_{\alpha}^{*}(f)\|_{L^{1}(\Sp)}+\|q(D)f\|_{L^{1}(\mathbb{R}^{n})}.
\]
\end{definition}

The following theorem is our main result.

\begin{theorem}\label{thm:main}
Let $\alpha>2$. Then
\[
\mathcal{H}_{FIO}^{1}(\mathbb{R}^{n})=\mathcal{H}_{FIO,G}^{1}(\mathbb{R}^{n})= \mathcal{H}_{FIO,\max}^{1}(\mathbb{R}^{n})=\mathcal{H}_{FIO,\mathcal{G}_{\alpha}^{*}}^{1}(\mathbb{R}^{n}),
\]
with equivalence of norms.
\end{theorem}

In particular, up to norm equivalence, the spaces in Theorem \ref{thm:main} are independent of the choice of functions $\theta_{\omega,\sigma}$, $\Phi$ and $\ph_{\w}$ with the required properties. 

Theorem \ref{thm:main} is proved in the main text as Theorems \ref{pro1}, \ref{pro3} and \ref{pro2}. In Section \ref{sec:applications} we give two applications of this result:
\begin{enumerate}
\item Theorem \ref{thm:czo}, which shows that a large class of operators which are bounded on $L^{p}(\Rn)$ for $1<p<\infty$ are also bounded on $\Hp$ for $1\leq p\leq\infty$;
\item Proposition \ref{prop:examples}, which determines in a relatively explicit manner the $\Hp$-norm of functions with frequency support in a dyadic-parabolic region.
\end{enumerate}
Our goal in this last section is to indicate how the equivalent characterizations of $\Hp$ can be used to incorporate techniques from other parts of harmonic analysis, and to calculate the $\HT^{p}_{FIO}(\Rn)$-norm of specific functions. In fact, the explicit description in Proposition \ref{prop:examples} of the $\Hp$-norm of a function with frequency support in a dyadic-parabolic region plays a crucial role in \cite{Rozendaal22b}, by connecting the Hardy spaces for FIOs to the $\ell^{p}$ decoupling inequality.

\subsection{Comparison to previous work}

In \cite{Rozendaal21} it is shown, for $1<p<\infty$, that an $f\in\Sw'(\Rn)$ satisfies $f\in{\Hp}$ if and only if $\ph_{\w}(D)\in L^{p}(\Rn)$ for almost all $\w\in S^{n-1}$, $\int_{S^{n-1}}\|\ph_{\w}(D)f\|_{L^{p}(\Rn)}^{p}\ud \w<\infty$, and $q(D)f\in L^{p}(\Rn)$. Moreover, in this case one has
\begin{equation}\label{eq:characterizationp}
\|f\|_{\Hp}\eqsim \bigg(\int_{S^{n-1}}\|\ph_{\w}(D)f\|_{L^{p}(\Rn)}^{p}\ud \w\bigg)^{1/p}+\|q(D)f\|_{L^{p}(\Rn)}
\end{equation}
for an implicit constant independent of $f$. Using classical characterizations of $L^{p}(\Rn)$ in terms of Littlewood--Paley g functions and maximal functions, one obtains from this similar characterizations of $\Hp$ for $1<p<\infty$ as are given in Theorem \ref{thm:main} for $p=1$.

In fact, in \cite[Remark 4.3]{Rozendaal21} the following question is posed. If $f\in \Sw'(\Rn)$ satisfies $\ph_{\w}(D)\in H^{1}(\Rn)$ for almost all $\w\in S^{n-1}$, $\int_{S^{n-1}}\|\ph_{\w}(D)f\|_{H^{1}(\Rn)}\ud \w<\infty$, and $q(D)f\in L^{1}(\Rn)$, does one have $f\in\HT^{1}_{FIO}(\Rn)$ and
\[
\|f\|_{\mathcal{H}_{FIO}^{1}(\mathbb{R}^{n})} \lesssim\int_{S^{n-1}}\|\varphi_\omega(D)f\|_{H^1(\RN)} \,\ud\omega+\|q(D)f\|_{L^{1}(\mathbb{R}^{n})},
\]
for an implicit constant independent of $f$? The reverse inequality was shown to hold (see \eqref{eq:reverse}). Using classical characterizations of $H^{1}(\Rn)$ in terms of Littlewood--Paley g functions, it is straightforward to show (see Proposition \ref{prop:equivalent}) that Theorem \ref{thm:main} gives an affirmative answer to this question. We leave as an open problem the question whether a similar characterization also holds for $p=\infty$ (see Remark \ref{rem:pinfty}).

It should be noted that the techniques used in this article to prove Theorem \ref{thm:main} are quite different from those in \cite{Rozendaal21}, although we do use the parabolic frequency localizations which played a key role in \cite{Rozendaal21}. More precisely, the characterizations of $\Hp$ for $1<p<\infty$ are obtained in \cite{Rozendaal21} by showing that each $f\in\Hp$ satisfies $\ph_{\w}(D)\in L^{p}(\Rn)$ for almost all $\w\in S^{n-1}$, $\int_{S^{n-1}}\|\ph_{\w}(D)f\|_{L^{p}(\Rn)}^{p}\ud \w<\infty$, and $q(D)f\in L^{p}(\Rn)$, with
\begin{equation}\label{eq:reverse}
\bigg(\int_{S^{n-1}}\|\ph_{\w}(D)f\|_{L^{p}(\Rn)}^{p}\ud \w\bigg)^{1/p}+\|q(D)f\|_{L^{p}(\Rn)}\lesssim \|f\|_{\Hp}.
\end{equation}
After that one uses duality to obtain the reverse inequality.

In the terminology of the present article, this amounts to showing that $\Hp\subseteq \HT^{p}_{FIO,G}(\Rn)$, where $\HT^{p}_{FIO,G}(\Rn)$ is defined in an analogous manner as in Definition \ref{def:HardyG}, and then using duality to obtain the reverse inclusion. For $p=1$, where we are interested in the inclusion $\HT^{1}_{FIO,G}(\Rn)\subseteq\HT^{1}_{FIO}(\Rn)$, such an approach does not appear to work. This is because $\HT^{1}_{FIO}(\Rn)$ is not the dual of $\HT^{\infty}_{FIO}(\Rn)$, and also because the norm of $\HT^{\infty}_{FIO}(\Rn)$ is of a different nature than that of $\Hp$ for $p<\infty$, so that the techniques from \cite{Rozendaal21} do not apply there. Instead, we prove the inclusion $\HT^{1}_{FIO,G}(\Rn)\subseteq\HT^{1}_{FIO}(\Rn)$ directly, using e.g.~pointwise inequalities for a maximal function of Peetre type, as well as boundedness of the vector-valued Hardy--Littlewood maximal function. Our proof is motivated in part by arguments from \cite{BuPaTa96, BuPaTa97, Hu17}.

\subsection{Organization of this article}
In Section \ref{sec:preliminaries}, we recall some notation and background on the metric $d$ and the wave packets which are used to define $\HT^{1}_{FIO}(\Rn)$. In Section \ref{sec3} we then show that $\mathcal{H}_{FIO}^{1}(\mathbb{R}^{n})=\HT^{1}_{FIO,G}(\Rn)$, and in Section
\ref{sec:maximal} we derive from this that $\HT^{1}_{FIO}(\Rn)=\HT^{1}_{FIO,\max}(\Rn)$. Next, in Section \ref{sec:conical}, we show that $\HT^{1}_{FIO}(\Rn)=\HT^{1}_{FIO,{\mathcal G}_{\alpha}^{*}}(\Rn)$, thereby completing the proof of Theorem \ref{thm:main}. We conclude with Section \ref{sec:applications}, which contains two applications of our main result.

\section{Notation and preliminaries}\label{sec:preliminaries}
\setcounter{equation}{0}

\subsection{Notation}

The natural numbers are $\mathbb{N}=\{1,2,\ldots\}$, and $\Z_{+}:=\N\cup\{0\}$. Throughout, we fix $n\in\N$ with $n\geq 2$. For $\xi,\eta\in\Rn$ we write $\lb\xi\rb:=(1+|\xi|^{2})^{1/2}$ and $\lb \xi,\eta\rb:=\xi\cdot\eta$, and for $\xi\neq 0$ we set $\hat{\xi}:=\xi/|\xi|$. We use multi-index notation, where $\partial^{\alpha}_{\xi}=\partial^{\alpha_{1}}_{\xi_{1}}\ldots\partial^{\alpha_{n}}_{\xi_{n}}$ for $\xi=(\xi_{1},\ldots,\xi_{n})\in\Rn$ and $\alpha=(\alpha_{1},\ldots,\alpha_{n})\in\Z_{+}^{n}$.

The Schwartz class and the class of tempered distributions on $\Rn$ are denoted by $\mathcal{S}(\mathbb{R}^{n})$ and $\mathcal{S}^{\prime}(\mathbb{R}^{n})$, respectively. The Fourier transform of an $f\in\mathcal{S}^{\prime}(\mathbb{R}^{n})$ is denoted by $\mathcal{F}f$, and for $f\in L^{1}(\Rn)$ it is normalized as follows:
\[
\F f(\xi)=\int_{\Rn}e^{-i\xi\cdot x}f(x)\,\ud x\quad(\xi\in\Rn).
\]
For $m:\Rn\to\C$ a measurable function of temperate growth, $m(D)$ is the Fourier multiplier with symbol $m$.

The volume of a measurable subset $B$ of a measure space is denoted by $V(B)$. If $V(B)<\infty$, then for an integrable function $f:B\to\C$ we write
\begin{align*}
\fint_{B}f(x)\,\ud x:=\frac{1}{V(B)}\int_{B}f(x) \,\ud x.
\end{align*}
The indicator function of a set $E$ is denoted by $\mathbf{1}_{E}$. For $(X,\mu)$ a measure space and $p,q\in[1,\infty)$, we denote by $L^{p}(X;\ell^{q})$ the space of all sequences $\{f_{j}\}_{j\in\N}$ of measurable functions $f_{j}:X\to \C$, $j\in\N$, such that
\begin{align*}
\|\{f_{j}\}_{j\in \N}\|_{L^{p}(X;\ell^{q})}:=\bigg(\int_{X}\|\{f_{j}(x)\}_{j\in\N}\|_{\ell^{q}}^{p}\ud\mu(x)\bigg)^{1/p}<\infty.
\end{align*}
We write $f(s)\lesssim g(s)$ to indicate that $f(s)\leq Cg(s)$ for all $s$ and a constant $C\geq0$ independent of $s$, and similarly for $f(s)\gtrsim g(s)$ and $g(s)\eqsim f(s)$.

\subsection{A metric on the cosphere bundle}\label{subsec:metric}

In this subsection, we collect some background on the underlying metric measure space which will be considered throughout. The relevant metric arises from contact geometry, but for this article we will only need a few basic facts about it. For more details on the material presented here, see \cite[Section 2.1]{HaPoRo20}.

Throughout, we denote elements of the sphere $S^{n-1}$ by $\w$ or $\nu$, and we let $g_{S^{n-1}}$ be the standard Riemannian metric on $S^{n-1}$. Let $\Sp:=\Rn\times S^{n-1}$ be the cosphere bundle of $\mathbb{R}^{n}$, endowed with the standard measure $\ud x\ud\w$. The $1$-form $\alpha_{S^{n-1}}:=\hat{\xi}\cdot dx$ on $\Sp$ determines a contact structure on $\Sp$, the smooth distribution of codimension $1$ hypersurfaces of $T(\Sp)$ given by the kernel of $\alpha_{S^{n-1}}$. Then $(\Sp,\alpha_{S^{n-1}})$ is a contact manifold. Together, the product metric $dx^{2}+g_{S^{n-1}}$ and the contact form determine a sub-Riemannian metric $d$ on $\Sp$:
\begin{align}\label{juli}
d(x,\omega;y,v):={\rm \inf\limits_{\gamma}}\int_{0}^{1}|\gamma^{\prime}(s)| \, \ud s.
\end{align}
for $(x,\omega),(y,\nu)\in \Sp$. Here the infimum is taken over all piecewise $C^{1}$-curves $\gamma:[0,1]\rightarrow \Sp$ such that $\gamma(0)=(x,\omega)$, $\gamma(1)=(y,\nu)$ and $\alpha_{S^{n-1}}(\gamma^{\prime}(s))=0$ for almost all $s\in[0,1]$. Moreover, $|\gamma^{\prime}(s)|$ is the length of the vector $\gamma^{\prime}(s)$ with respect to $dx^{2}+dg_{S^{n-1}}$.

It is shown in \cite[Lemma 2.1]{HaPoRo20} that
\begin{align*}
d(x,\omega;y,\nu)\eqsim \big{(}|\langle\omega,x-y\rangle|+|x-y|^{2}+|\omega-\nu|^{2}\big{)}^{1/2}
\end{align*}
for an implicit constant independent of $(x,\w),(y,\nu)\in\Sp$. The following is \cite[Lemma 2.3]{HaPoRo20}.

\begin{lemma}\label{lem7}
There exists a constant $C>0$ such that, for all $(x,\omega)\in \Sp$, one has
\begin{align*}
\frac{1}{C}\tau^{2n}\leq V(B_{\tau}(x,\omega))\leq C\tau^{2n}
\end{align*}
if $\tau\in(0,1)$ and
\begin{align*}
\frac{1}{C}\tau^{n}\leq V(B_{\tau}(x,\omega))\leq C\tau^{n}
\end{align*}
if $\tau\geq 1$. In particular,
\begin{align*}
V(B_{\lambda \tau}(x,\omega))\leq C\lambda^{2n}V(B_{\tau}(x,\omega))
\end{align*}
for all $\tau>0$ and $\lambda \geq 1$, and $(\Sp,d,\ud x\ud\omega)$ is a doubling metric measure space.
\end{lemma}

\subsection{Wave packets}\label{sub1}

In this subsection we introduce the wave packets which are used to define the Hardy spaces for Fourier integral operators. For more on this, see \cite[Section 4]{HaPoRo20} and \cite[Section 3]{Rozendaal21}.

Fix a non-negative radial $\varphi \in C_{c}^{\infty}(\mathbb{R}^{n})$ such that $\varphi\equiv 1$ in a neighborhood of zero and $\varphi(\xi)=0$ for $|\xi|>1$. For $\sigma>0$, $\w\in S^{n-1}$ and $\xi\in\Rn\setminus\{0\}$ set $c_{\sigma}:=\big{(}\int_{S^{n-1}}\varphi(\frac{e_{1}-\nu}{\sqrt{\sigma}})^{2}\ud\nu\big{)}^{-1/2}$, where $e_{1}$ is the first basis vector of $\Rn$ (this particular choice is irrelevant), and $\varphi_{\omega, \sigma}(\xi):=c_{\sigma}\varphi\big(\frac{\hat{\xi}-\omega}{\sqrt{\sigma}}\big)$. Also let $\varphi_{\omega, \sigma}(0):=0$. Next, let $\Psi\in \mathcal{S}(\mathbb{R}^{n})$ be a non-negative radial function, with $\Psi(\xi)=0$ if $|\xi|\notin[\frac{1}{2},2]$, $\Psi(\xi)=c>0$ if $|\xi|\in [\frac{3}{4},\frac{3}{2}]$, and
\begin{equation}\label{eq:Psi}
\int_{0}^{\infty}\Psi(\sigma\xi)^{2}\frac{\ud\sigma}{\sigma}=1\quad (\xi\neq 0).
\end{equation}
For $\sigma>0$ and $\xi\in\Rn$ set $\Psi_{\sigma}(\xi):=\Psi(\sigma\xi)$. Now, for $\omega\in S^{n-1}$, write
\[
\varphi_{\omega}(\xi):=\int_{0}^{4}\Psi_{\tau}(\xi)\ph_{\w,\tau}(\xi)\frac{\ud\tau}{\tau}
\]
and, if $\sigma\in(0,1)$,
\begin{equation}\label{the def of theta}
\theta_{\omega,\sigma}(\xi):=\Psi_{\sigma}(\xi)\varphi_{\omega}(\xi).
\end{equation}
These wave packets were introduced in \cite{Rozendaal21}, and in this article they have already appeared in Definitions \ref{Area function}, \ref{def:HardyG} and \ref{def:Hardyothers}.

We also introduce some new wave packets. Set
\begin{equation}\label{eq:eta}
\eta(\xi):=\begin{cases}\frac{\Psi(\xi)}{\sum_{j\in\mathbb{Z}}\Psi(2^{-j}\xi)^{2}} &\text{for }\xi\neq 0, \\0 &\text{for }\xi=0, \end{cases}
\end{equation}
and, for $\w\in S^{n-1}$ and $0<\sigma<1$,
\begin{equation}\label{eq:chi}
\chi_{\omega,\sigma}(\xi):=\begin{cases}\frac{\eta(\sigma\xi)\varphi_{\omega}(\xi)}{\int_{S^{n-1}}\varphi_{\nu}(\xi)^{2}\ud\nu} &\text{ for }\xi\in {\rm supp}(\theta_{\omega,\sigma}),\\0 &\text{otherwise}.\end{cases}
\end{equation}
We collect some properties of these wave packets in the following lemma.

\begin{lemma}\label{WP}
For $\omega\in S^{n-1}$ and $0<\sigma<1$, let $\gamma_{\omega,\sigma}\in\{\theta_{\omega,\sigma},\chi_{\omega,\sigma}\}$. Then $\gamma_{\omega,\sigma}\in C_{c}^{\infty}(\mathbb{R}^{n})$, and
\begin{equation}\label{eq:dyadicpar}
{\rm supp}(\gamma_{\omega,\sigma})\subseteq\big\{\xi\in\Rn\mid \tfrac{1}{2}\sigma^{-1}\leq |\xi|\leq 2\sigma^{-1}, |\hat\xi-\omega|\leq 2\sqrt{\sigma}\big\}.
\end{equation}
Moreover, for all $\alpha\in \mathbb{Z}_{+}^{n}$ and $\beta\in \mathbb{Z}_{+}$, there exists a constant $C_{\alpha,\beta}\geq 0$ such that
\begin{align}\label{daoshu}
|\langle\omega,\nabla_{\xi}\rangle^{\beta}\partial_{\xi}^{\alpha}\gamma_{\omega,\sigma}(\xi)|\leq C_{\alpha,\beta}\sigma^{-\frac{n-1}{4}+\frac{|\alpha|}{2}+\beta}
\end{align}
for all $(\xi,\omega,\sigma)\in \Sp\times (0,1)$. For each $N\geq 0$, there exists a constant $C_{N}\geq 0$ such that
\begin{align}\label{fourier}
|\mathcal{F}^{-1}(\gamma_{\omega,\sigma})(x)|\leq C_{N}\sigma^{-\frac{3n+1}{4}}(1+\sigma^{-1}|x|^{2}+\sigma^{-2}\langle\omega,x\rangle^{2})^{-N}
\end{align}
for all $(x,\omega,\sigma)\in \Sp\times (0,1)$. Finally, for all $\alpha\in\Z_{+}^{n}$ there exists a constant $C_{\alpha}\geq0$ such that
\begin{equation}\label{eq:phinu3}
\Big|\partial_{\xi}^{\alpha}\bigg(\int_{S_{n-1}}\ph_{\nu}(\xi)\ud \nu\bigg)^{-1}\Big|\leq C_{\alpha}|\xi|^{\frac{n-1}{4}-|\alpha|}
\end{equation}
for all $\xi\in\Rn$ with $|\xi|\geq1/2$.
\end{lemma}
\begin{proof}
For $\gamma_{\omega,\sigma}=\theta_{\omega,\sigma}$, the required statements are contained in \cite[Lemma 3.2]{Rozendaal21}. It is also shown there (see \cite[Remark 3.3]{Rozendaal21} and the arguments for \eqref{eq:phinu3} below) that, for all $\alpha\in\mathbb{Z}_{+}^{n}$ and $\beta\in\mathbb{Z}_{+}$, there exists constants $C'_{\alpha},C'_{\alpha,\beta}\geq0$ such that
\begin{equation}\label{eq:phinu1}
\Big|\partial_{\xi}^{\alpha}\bigg(\int_{S^{n-1}}\varphi_{\nu}(\xi)^{2}\ud\nu\bigg)^{-1}\Big|\leq C_{\alpha}'\sigma^{|\alpha|}
\end{equation}
and
\begin{equation}\label{eq:phinu2}
|\langle\omega,\nabla_{\xi}\rangle^{\beta}\partial_{\xi}^{\alpha}\varphi_{\omega}(\xi)|\leq C_{\alpha,\beta}'\sigma^{-\frac{n-1}{4}+\frac{|\alpha|}{2}+\beta}
\end{equation}
for all $\w\in S^{n-1}$, $0<\sigma<1$ and $\xi\in\supp(\theta_{\omega,\sigma})$.

For $\gamma_{\omega,\sigma}=\chi_{\omega,\sigma}$, we first use the properties of $\Psi$ to note that
\[
\sum_{j\in\Z}\Psi(2^{-j}\xi)^{2}\geq c
\]
for all $\xi\neq 0$, since there exists a $j\in\Z$ such that $2^{-j}|\xi|\in [\frac{3}{4},\frac{3}{2}]$. In turn, this implies that $\eta$ is well defined, and it is straightforward to see that in fact $\eta\in C^{\infty}_{c}(\Rn)$. It now follows that $\chi_{\w,\sigma}\in C^{\infty}_{c}(\Rn)$ is well defined with $\supp(\chi_{\w,\sigma})=\supp(\theta_{\w,\sigma})$. Moreover, clearly
\[
|\partial_{\xi}^{\alpha}\eta(\sigma\xi)|=\sigma^{|\alpha|}|(\partial_{\xi}^{\alpha}\eta)(\sigma\xi)|\lesssim\sigma^{|\alpha|}
\]
for all $\alpha\in\Z_{+}^{n}$, with an implicit constant independent of $\sigma>0$ and $\xi\in\Rn$. By combining this with \eqref{eq:phinu1} and \eqref{eq:phinu2}, it follows that $\chi_{\w,\sigma}$ satisfies \eqref{daoshu}. For \eqref{fourier} one now integrates by parts with respect to the operator
\begin{align*}
L:=\big(1+\sigma^{-1}|x|^{2}+\sigma^{-2}\langle\omega,x\rangle^{2}\big)^{-1}\big(1-\sigma^{-1}\Delta_{\xi}-\sigma^{-2}\langle\omega,\nabla_{\xi}\rangle^{2}\big)
\end{align*}
in the expression
\begin{align*}
\mathcal{F}^{-1}(\chi_{\omega,\sigma})(x):=\frac{1}{(2\pi)^{n}}\int_{\mathbb{R}^{n}}e^{ix\cdot\xi}{\chi}_{\omega,\sigma}(\xi)\,\ud\xi\quad(x\in \mathbb{R}^{n}),
\end{align*}
using \eqref{daoshu} and the support properties of $\chi_{\w,\sigma}$. See \cite[Lemma 4.1]{HaPoRo20} for more details.

Finally, \eqref{eq:phinu3} is obtained in the same manner as \eqref{eq:phinu1}. 
\end{proof}

We will also need the following corollary. The estimates in \eqref{eq:off} were called off-singularity bounds in \cite{HaPoRo20}, and they are useful for showing that an operator is bounded on $\HT^{p}_{FIO}(\Rn)$, for $1\leq p\leq \infty$.

\begin{coro}\label{off}
For $w,\nu\in S^{n-1}$ and $\sigma,\tau\in(0,1)$, let $K_{\sigma,\tau}^{\w,\nu}$ be the integral kernel associated with the operator
\[
f\mapsto \theta_{\omega,\sigma}(D)\chi_{\nu,\tau}(D)f
\]
on $\Sw(\Rn)$. Then for each $N\geq 0$ there exists a $C_{N}\geq 0$, independent of $\w$, $\nu$, $\sigma$ and $\tau$, such that
\begin{equation}\label{eq:off}
|K_{\sigma,\tau}^{\w,\nu}(x,y)|\leq C_{N}\min\Big(\frac{\sigma}{\tau},\frac{\tau}{\sigma}\Big)^{N}\rho^{-n}(1+\rho^{-1}d(x,\omega;y,\nu)^{2})^{-N}
\end{equation}
for all $x,y\in\Rn$, where $\rho=\min(\sigma,\tau)$.
\end{coro}
\begin{proof}
To obtain \eqref{eq:off}, it suffices to repeat the arguments in \cite[Proposition 3.6]{Rozendaal21} (see also \cite[Remark 3.7]{Rozendaal21} and \cite[Theorem 5.1]{HaPoRo20}), which rely only on integration by parts and on the properties of the wave packets in Lemma \ref{WP}.
\end{proof}

\section{The Littlewood--Paley g function characterization}\label{sec3}

This section is devoted to showing that $\HT^{1}_{FIO}(\Rn)=\HT^{1}_{FIO,G}(\Rn)$. By \cite[Proposition 2.1 and Remark 2.2]{AuHoMa12} (see also \cite[Equation (2.9)]{HaPoRo20}) one has $\HT^{1}_{FIO}(\Rn)\subseteq \HT^{1}_{FIO,G}(\Rn)$, so it suffices to show that $\mathcal{H}_{FIO,G}^{1}(\mathbb{R}^{n})\subseteq \mathcal{H}_{FIO}^{1}(\mathbb{R}^{n})$. To do so, we first collect some preliminary results which will be used to prove the required embedding.

\subsection{Preliminary results}\label{subsec:preliminary}

In this subsection we first prove a useful equivalent characterization of $\HT^{1}_{FIO,G}(\Rn)$, from which we derive a Sobolev embedding for $\HT^{1}_{FIO,G}(\Rn)$. Then we prove a technical lemma which will be used afterwards to obtain a pointwise inequality for a maximal function of Peetre type. This maximal function will in turn play a crucial role in the proof of the main result of this section.

\begin{proposition}\label{prop:equivalent}
An $f\in\Sw'(\Rn)$ satisfies $f\in \mathcal{H}_{FIO,G}^{1}(\mathbb{R}^{n})$ if and only if $q(D)f\in L^{1}(\Rn)$, $\ph_{\w}(D)f\in H^{1}(\Rn)$ for almost all $\w\in S^{n-1}$, and
\[
\int_{S^{n-1}}\|\varphi_{\omega}(D)f\|_{H^{1}(\mathbb{R}^{n})}\,\ud \omega<\infty.
\]
Moreover, in this case one has
\[
\|f\|_{\mathcal{H}_{FIO,G}^{1}(\mathbb{R}^{n})}\eqsim \int_{S^{n-1}}\|\varphi_{\omega}(D)f\|_{H^{1}(\mathbb{R}^{n})}\,\ud \omega+\|q(D)f\|_{L^{1}(\mathbb{R}^{n})}.
\]
\end{proposition}
\begin{proof}
By the Littlewood--Paley g function characterization of $H^{1}(\Rn)$ (see \cite{Stein93} or \cite{Uchiyama85}), it suffices to show that an $f\in\Sw'(\Rn)$ satisfies $f\in \mathcal{H}_{FIO,G}^{1}(\mathbb{R}^{n})$ if and only if $q(D)f\in L^{1}(\Rn)$ and $G'(f)\in L^{1}(\Sp)$, with
\[
\|f\|_{\mathcal{H}_{FIO,G}^{1}(\mathbb{R}^{n})}\eqsim \|G'(f)\|_{L^{1}(\Sp)}+\|q(D)f\|_{L^{1}(\Rn)}.
\]
Here
\[
G'(f)(x,\w):=\bigg(\int_{0}^{\infty}|\theta_{\w,\sigma}(D)f(x)|^{2}\frac{\ud\sigma}{\sigma}\bigg)^{1/2}
\]
for $(x,\w)\in\Sp$. In turn, since $G(f)\leq G'(f)$ pointwise, it suffices to prove that each $f\in\HT^{1}_{FIO,G}(\Rn)$ satisfies $G'(f)\in L^{1}(\Sp)$ and $\|G'(f)\|_{L^{1}(\Sp)}\lesssim \|f\|_{\HT^{1}_{FIO,G}(\Rn)}$.

Let $f\in\HT^{1}_{FIO,G}(\Rn)$ and note that
\[
\theta_{\w,\sigma}(\xi)=\Psi_{\sigma}(\xi)\ph_{\w}(\xi)=\int_{0}^{4}\Psi_{\sigma}(\xi)\Psi_{\tau}(\xi)\ph_{\w,\tau}(\xi)\frac{\ud\tau}{\tau}=0
\]
for all $\xi\in\Rn$ if $\sigma>16 $, since for all $\tau>0$ one has $\Psi_{\tau}(\xi)=0$ if $|\xi|\notin[\tau^{-1}/2,2\tau^{-1}]$. Hence one in fact has
\[
G'(f)(x,\w)=\bigg(\int_{0}^{16}|\theta_{\w,\sigma}(D)f(x)|^{2}\frac{\ud\sigma}{\sigma}\bigg)^{1/2}\leq G(f)(x,\w)+\bigg(\int_{1}^{16}|\theta_{\w,\sigma}(D)f(x)|^{2}\frac{\ud\sigma}{\sigma}\bigg)^{1/2}
\]
for all $(x,\w)\in\Sp$, and it suffices to show that
\[
\int_{\Sp}\bigg(\int_{1}^{16}|\theta_{\w,\sigma}(D)f(x)|^{2}\frac{\ud\sigma}{\sigma}\bigg)^{1/2}\ud x\ud\w\lesssim \|f\|_{\HT^{1}_{FIO,G}(\Rn)}.
\]
But this is proved by noting that $\theta_{\w,\sigma}(D)(1-q)(D)f=0$ for $\sigma>1$, since $q(\xi)=1$ for $|\xi|\leq 2$, and then reasoning as follows:
\begin{align*}
&\int_{\Sp}\bigg(\int_{1}^{16}|\theta_{\w,\sigma}(D)f(x)|^{2}\frac{\ud\sigma}{\sigma}\bigg)^{1/2}\ud x\ud\w=\int_{\Sp}\bigg(\int_{1}^{16}|\theta_{\w,\sigma}(D)q(D)f(x)|^{2}\frac{\ud\sigma}{\sigma}\bigg)^{1/2}\ud x\ud\w\\
&\lesssim \int_{\mathbb{R}^{n}}\sup_{\frac{1}{2}\leq\sigma\leq 16,\omega\in S^{n-1}}|\theta_{\w,\sigma}(D)q(D)f(x)|\ud x\lesssim \int_{\mathbb{R}^{n}}\int_{\mathbb{R}^{n}}\frac{|q(D)f(y)|}{(1+|x-y|)^{n+1}}\,\ud y\ud x\\
&\lesssim \|q(D)f\|_{L^{1}(\mathbb{R}^{n})}\leq \|f\|_{\HT^{1}_{FIO,G}(\Rn)}.
\end{align*}
Note that the bounds for $\F^{-1}(\theta_{\w,\sigma})$ that we used in the penultimate line are contained in \eqref{fourier}.
\end{proof}

We can now derive a useful Sobolev embedding for $\HT^{1}_{FIO,G}(\Rn)$, which is formulated in terms of the local real Hardy space $\HT^{1}(\Rn)$ defined by Goldberg \cite{Goldberg79}. Choose a function $r\in C^{\infty}_{c}(\Rn)$ such that $r(\xi)=1$ if $|\xi|\leq 1$. Then $\HT^{1}(\Rn)$ consists of all $f\in\Sw'(\Rn)$ such that $r(D)f\in L^{1}(\Rn)$ and $(1-r)(D)f\in H^{1}(\Rn)$, with the norm
\begin{equation}\label{eq:localhardy}
\|f\|_{\HT^{1}(\Rn)}:= \|r(D)f\|_{L^{1}(\Rn)}+\|(1-r)(D)f\|_{H^{1}(\Rn)}.
\end{equation}
Up to norm equivalence, this definition does not depend on the specific choice of $r$.

\begin{proposition}\label{prop:Sobolev}
The map $\lb D\rb^{-\frac{n-1}{4}}:\mathcal{H}_{FIO,G}^{1}(\mathbb{R}^{n})\to \HT^{1}(\Rn)$ is bounded. Hence $\mathcal{H}_{FIO,G}^{1}(\mathbb{R}^{n})\subseteq W^{-\frac{n-1}{4},1}(\Rn)$.
\end{proposition}
\begin{proof}
For the first statement we let $r:=q$ and fix $f\in\HT^{1}_{FIO,G}(\Rn)$. Then $q(D)\lb D\rb^{-\frac{n-1}{4}}f\in L^{1}(\Rn)$ with
\[
\|q(D)\lb D\rb^{-\frac{n-1}{4}}f\|_{L^{1}(\Rn)}\lesssim \|q(D)f\|_{L^{1}(\Rn)}\leq \|f\|_{\HT^{1}_{FIO,G}(\Rn)}.
\]
To show that $(1-q)(D)\lb D\rb^{-\frac{n-1}{4}} f\in H^{1}(\Rn)$, define $m\in C^{\infty}(\Rn)$ by
\begin{equation}\label{eq:defm}
m(\xi):=\begin{cases}
(1-q(\xi))\lb \xi\rb^{-\frac{n-1}{4}}\big(\int_{S^{n-1}}\ph_{\nu}(\xi)\ud \nu\big)^{-1}&\text{if }|\xi|\geq 1/2,\\
0&\text{otherwise}.
\end{cases}
\end{equation}
It follows from \eqref{eq:phinu3} that $m(D):H^{1}(\Rn)\to H^{1}(\Rn)$ is continuous, and one has
\begin{equation}\label{eq:reproform}
(1-q)(D)\lb D\rb^{-\frac{n-1}{4}}f=\int_{S^{n-1}}m(D)\ph_{\w}(D)f\ud\w
\end{equation}
since $q(\xi)=1$ if $|\xi|\leq 2$. Hence $(1-q)(D)\lb D\rb^{-\frac{n-1}{4}}f\in H^{1}(\Rn)$ with
\begin{align*}
\|(1-q)(D)\lb D\rb^{-\frac{n-1}{4}}f\|_{H^{1}(\Rn)}&=\Big\|\int_{S^{n-1}}m(D)\ph_{\w}(D)f\ud\w\Big\|_{H^{1}(\Rn)}\\
&\leq\int_{S^{n-1}}\|m(D)\ph_{\w}(D)f\|_{H^{1}(\Rn)}\ud\w\lesssim \int_{S^{n-1}}\|\ph_{\w}(D)f\|_{H^{1}(\Rn)}\ud\w\lesssim \|f\|_{\HT^{1}_{FIO,G}(\Rn)},
\end{align*}
where for the final inequality we used Proposition \ref{prop:equivalent}.

The second statement of the proposition now follows from the inclusion $H^{1}(\Rn)\subseteq L^{1}(\Rn)$.
\end{proof}

\begin{remark}\label{rem:Hardy}
The same embedding as in Proposition \ref{prop:Sobolev} was obtained for $\HT^{1}_{FIO}(\Rn)$ in \cite[Theorem 7.4]{HaPoRo20}, with a somewhat similar proof. However, we cannot appeal to that result here since we have not yet shown that $\HT^{1}_{FIO,G}(\Rn)\subseteq\HT^{1}_{FIO}(\Rn)$ (and in fact we will use Proposition \ref{prop:Sobolev} to prove this inclusion).
\end{remark}

We will also need the following technical lemma.

\begin{lemma}\label{lem1}
Let $0<r\leq1$, and let $\{b_{l}\}_{l=1}^{\infty}\subseteq[0,\infty]$ and $\{d_{l}\}_{l=1}^{\infty}\subseteq[0,\infty)$ be two sequences. Assume that there exist $C_{0},N_{0}>0$ such that
\begin{equation}\label{a2}
d_{l}\leq C_{0}2^{lN_{0}}\quad(l\in\N),
\end{equation}
and that for each $N>N_{0}$ there exists a $C_{N}>0$ such that
\begin{equation}\label{eq:a2a}
d_{l}\leq C_{N} \sum_{j=1}^{\infty}2^{-|j-l|N}b_{j}d_{j}^{1-r}\quad(l\in\N).
\end{equation}
Then
\[
d_{l}^{r}\leq C_{N} \sum_{j=1}^{\infty}2^{-|j-l|Nr}b_{j}\quad(l\in\N).
\]
\begin{proof}
The proof of Lemma \ref{lem1} is essentially contained in \cite{Rychkov99}, but for the reader's convenience we give a simple proof here. Without loss of generality, we may assume that $\{d_{l}\}_{l=1}^{\infty}$ is not the zero sequence, and then \eqref{a2} shows that $D_{l,N}:=\sup_{k\in\N}2^{-|l-k|N}d_{k}\in(0,\infty)$ for all $l\in\N$ and $N>N_{0}$. Now \eqref{eq:a2a} yields
\begin{align*}
D_{l,N}
&\leq
\sup_{k\in\N}2^{-|l-k|N}C_{N}\sum_{j=1}^{\infty}2^{-|j-k|N}b_{j}d_{j}^{1-r}\leq
C_{N}\sum_{j = 1}^{\infty}2^{-|j-l|N} b_{j}d_{j}^{1-r}\\
&\leq
C_{N}\sum_{j = 1}^{\infty}2^{-|j-l|N}b_{j}2^{|j-l|N(1-r)}D_{l,N}^{1-r}=
C_{N}\sum_{j = 1}^{\infty}2^{-|j-l|Nr}b_{j}D_{l,N}^{1-r}
\end{align*}
for all $j\in\N$. Multiplying by $D_{l,N}^{r-1}$, we obtain from this the required conclusion:
\[
d_{l}^{r}\leq D_{l,N}^{r}\leq C_{N}\sum_{j = 1}^{\infty}2^{-|j-l|Nr}b_{j}.\qedhere
\]
\end{proof}
\end{lemma}

For the main result of this section we will work with a Peetre type maximal function. For $\alpha>0$, $f\in \mathcal{S^{\prime}}(\mathbb{R}^{n})$ and $(x,\omega,\sigma)\in \Sp\times (0,\infty)$, set
\begin{align*}
M_{\alpha}^{*}(f)(x,\omega,\sigma):= \sup_{(y,\nu)\in \Sp}\frac{|\theta_{\nu,\sigma}(D)f(y)|}{(1+\sigma^{-1}d(x,\omega;y,\nu)^{2})^{\alpha}},
\end{align*}
where the metric $d$ on $\Sp$ is as in Section \ref{subsec:metric}. We will apply Lemma \ref{lem1} to a sequence arising from this maximal function, and in the following lemma we show that the growth condition \eqref{a2} is satisfied for this sequence.

\begin{lemma}\label{lem2}
Let $\alpha>0$. Then there exists a $C_{\alpha}>0$ with the following property. For all $f\in W^{-\frac{n-1}{4},1}(\mathbb{R}^{n})$, $(x,\w)\in\Sp$, $l\in\N$ and $\sigma\in(1,2)$, one has
\[
M_{\alpha}^{*}(f)(x,\omega,2^{-l}\sigma)\leq C_{\alpha}2^{ln}\|f\|_{W^{-\frac{n-1}{4},1}(\mathbb{R}^{n})}.
\]
\begin{proof}
Fix $f\in W^{-\frac{n-1}{4},1}(\mathbb{R}^{n})$, $(x,\w)\in\Sp$, $l\in\N$ and $\sigma\in(1,2)$. For $\tau\in(0,1)$ and $\xi\in\Rn$, set $\tilde{\theta}_{\omega,\tau}(\xi):=\tau^{\frac{n-1}{4}}\langle \xi\rangle ^{\frac{n-1}{4}}\theta_{\omega,\tau}(\xi)$. It is straightforward to see that $\tilde{\theta}_{\w,\tau}\in C^{\infty}_{c}(\Rn)$, with the same support properties and upper bounds as $\theta_{\w,\tau}$ from Lemma \ref{WP}, with constants independent of $\tau$. In particular, using \eqref{daoshu}, we obtain
\begin{align*}
M_{\alpha}^{*}(f)(x,\omega,2^{-l}\sigma)
&=\sup_{(y,\nu)\in \Sp}\frac{|\theta_{\nu,2^{-l}\sigma}(D)f(y)|}{(1+2^{l}\sigma^{-1}d(x,\omega;y,\nu)^{2})^{\alpha}}\leq \sup_{(y,\nu)\in \Sp}|\theta_{\nu,2^{-l}\sigma}(D)f(y)|\\
&\leq 2^{\frac{n-1}{4}l} \sup_{(y,\nu)\in \Sp}|\tilde{\theta}_{\nu,2^{-l}\sigma}(D)\langle D\rangle ^{-\frac{n-1}{4}}f(y)|\\
&\leq 2^{\frac{n-1}{4}l}\sup_{(y,\nu)\in \Sp}\int_{\Rn}|\F^{-1}(\tilde{\theta}_{\nu,2^{-l}\sigma})(y-z)\lb D\rb^{-\frac{n-1}{4}}f(z)|\,\ud z\\
&\lesssim 2^{nl}\int_{\mathbb{R}^{n}}|\langle D\rangle^{-\frac{n-1}{4}}f(z)|\,\ud z=2^{nl}\|f\|_{W^{-\frac{n-1}{4},1}(\mathbb{R}^{n})}.\qedhere
\end{align*}
\end{proof}
\end{lemma}

Having verified the conditions of Lemma \ref{lem1}, we can now apply this lemma to obtain a useful inequality for our maximal function.

\begin{proposition}\label{lem4}
Let $\alpha>0$ and $r\in(0,1)$. Then for each $N>0$ there exists a $C_{\alpha,r,N}>0$ such that, for all $\sigma\in(1,2)$, $l\in\N$ and $f\in W^{-\frac{n-1}{4},1}(\mathbb{R}^{n})$ with $\F f(\xi)=0$ for $|\xi|\leq 2$, one has
\[
[M_{\alpha}^{*}(f)(x,\omega,2^{-l}\sigma)]^{r}\leq C_{\alpha,r,N}\sum_{j=1}^{\infty}2^{-|j-l|N}\int_{\Sp}2^{ln}(1+2^{l}d(x,\omega;z,\mu)^{2})^{-\alpha r}|\theta_{\mu,2^{-j}\sigma}(D)f(z)|^{r}\ud z\ud\mu.
\]
\end{proposition}
Note that the Fourier transform $\F f$ of an $f\in W^{-\frac{n-1}{4},1}(\Rn)$ is a function of at most polynomial growth, so the pointwise condition $\F f(\xi)=0$ for $|\xi|\leq 2$ is well defined. We also note that the assumption $f\in W^{-\frac{n-1}{4},1}(\Rn)$ can be extended to $f\in W^{s,1}(\Rn)$ for some $s\in\R$, but we will not need such generality in the remainder.
\begin{proof}
Clearly we may consider $N\geq \alpha$. Fix $(x,\w)\in\Sp$, $\sigma\in(1,2)$, $l\in\N$ and $f\in W^{-\frac{n-1}{4},1}(\Rn)$ with $\F f(\xi)=0$ for $|\xi|\leq 2$. Recall from the proof of Lemma \ref{WP} that $\eta$, as defined in \eqref{eq:eta}, satisfies $\eta\in C^{\infty}_{c}(\Rn)$ and $\supp(\eta)=\supp(\Psi)\subseteq \{\xi\in\mathbb{R}^{n}:\frac{1}{2}\leq|\xi|\leq2\}$. As we did for $\Psi$, write $\eta_{\tau}(\xi):=\eta(\tau\xi)$ for $\tau>0$ and $\xi\in\Rn$. Then, by definition, the following identity holds for $\xi\neq0$:
\[
\sum_{j\in \mathbb{Z}}\eta_{2^{-j}\sigma}(\xi)\Psi_{2^{-j}\sigma}(\xi)=1.
\]
Now, by the assumption on the support of $\F f$ and because $\sigma\in(1,2)$, one has $\Psi_{2^{-j}\sigma}(D)f=0$ for $j\leq 0$. Hence, using the definition of $\chi_{\mu,2^{-j}\sigma}$ from \eqref{eq:chi}, a direct calculation yields
\begin{align*}
\theta_{\nu,2^{-l}\sigma}(D)f(y)&=\sum_{j\in\Z}\theta_{\nu,2^{-l}\sigma}(D)\eta_{2^{-j}\sigma}(D)\Psi_{2^{-j}\sigma}(D)f(y)=\sum_{j=1}^{\infty}\theta_{\nu,2^{-l}\sigma}(D)\eta_{2^{-j}\sigma}(D)\Psi_{2^{-j}\sigma}(D)f(y)\\
&=\sum_{j=1}^{\infty}\int_{S^{n-1}}\theta_{\nu,2^{-l}\sigma}(D)\chi_{\mu,2^{-j}\sigma}(D)\theta_{\mu,2^{-j}\sigma}(D)f(y)\,\ud\mu
\end{align*}
for all $(y,\nu)\in\Sp$. Now apply Corollary \ref{off} to
\[
K_{2^{-l}\sigma,2^{-j}\sigma}^{\nu,\mu}(y,z)=\F^{-1}(\theta_{\nu,2^{-l}\sigma}\chi_{\mu,2^{-j}\sigma})(y-z)=\frac{1}{(2\pi)^{n}}\int_{\mathbb{R}^{n}}e^{i\langle y-z,\xi\rangle}\theta_{\nu,2^{-l}\sigma}(\xi)\chi_{\mu,2^{-j}\sigma}(\xi)\,\ud\xi,
\]
for $(z,\mu) \in \Sp$ and $j\in\N$, to obtain
\begin{align*}
|\theta_{\nu,2^{-l}\sigma}(D)f(y)|
&\leq \sum_{j=1}^{\infty}\Big|\int_{S^{n-1}}\theta_{\nu,2^{-l}\sigma}(D)\chi_{\mu,2^{-j}\sigma}(D)\theta_{\mu,2^{-j}\sigma}(D)f(y)\,\ud\mu\Big|\\
&\leq \sum_{j=1}^{\infty}\int_{S^{n-1}}\int_{\mathbb{R}^{n}}|K_{2^{-l}\sigma,2^{-j}\sigma}^{\nu,\mu}(y,z)\theta_{\mu,2^{-j}\sigma}(D)f(z)|\,\ud z\ud\mu\\
&\lesssim\sum_{j=1}^{\infty}2^{-|j-l|N}\int_{\Sp}2^{ln}(1+2^{l}d(y,\nu;z,\mu)^{2})^{-N}|\theta_{\mu,2^{-j}\sigma}(D)f(z)|\,\ud z\ud\mu.
\end{align*}
In turn, we can use that $d$ is a metric and that $N\geq\alpha$ to derive from this that
\begin{align*}
&M_{\alpha}^{*}(f)(x,\omega,2^{-l}\sigma)\lesssim\sup_{(y,\nu)\in \Sp}\frac{|\theta_{\nu,2^{-l}\sigma}(D)f(y)|}{(1+2^{l}d(x,\omega;y,\nu)^{2})^{\alpha}}\\
&\lesssim \sup_{(y,\nu)\in \Sp}\sum_{j=1}^{\infty}2^{-|j-l|N}\int_{\Sp}2^{ln}(1+2^{l}d(y,\nu;z,\mu)^{2})^{-N}(1+2^{l}d(x,\omega;y,\nu)^{2})^{-\alpha}|\theta_{\mu,2^{-j}\sigma}(D)f(z)|\,\ud z\ud\mu\\
&\leq \sum_{j=1}^{\infty}2^{-|j-l|N}\int_{\Sp}2^{ln}(1+2^{l}d(x,\omega;z,\mu)^{2})^{-\alpha}|\theta_{\mu,2^{-j}\sigma}(D)f(z)|\,\ud z\ud\mu\\
&= \sum_{j=1}^{\infty}2^{-|j-l|N}\int_{\Sp}2^{ln}(1+2^{l}d(x,\omega;z,\mu)^{2})^{-\alpha}|\theta_{\mu,2^{-j}\sigma}(D)f(z)|^{r}|\theta_{\mu,2^{-j}\sigma}(D)f(z)|^{1-r}\ud z\ud\mu\\
&\leq \sum_{j=1}^{\infty}2^{-|j-l|N}\int_{\Sp}2^{ln}\frac{(1+2^{j}d(x,\omega;z,\mu)^{2})^{\alpha(1-r)}}{(1+2^{l}d(x,\omega;z,\mu)^{2})^{\alpha}}|\theta_{\mu,2^{-j}\sigma}(D)f(z)|^{r}\ud z\ud\mu (M_{\alpha}^{*}(f)(x,\omega,2^{-j}\sigma))^{1-r}.
\end{align*}
Moreover, since one has
\[
1+2^{j}d(x,\omega;z,\mu)^{2}\leq (1+2^{l}d(x,\omega;z,\mu)^{2})2^{|l-j|}
\]
for all $j\in\N$ and $(z,\mu)\in\Sp$, we can write
\begin{align*}
&M_{\alpha}^{*}(f)(x,\omega,2^{-l}\sigma)\\
&\lesssim \sum_{j=1}^{\infty}2^{-|j-l|N}\int_{\Sp}2^{ln}\frac{(1+2^{j}d(x,\omega;z,\mu)^{2})^{\alpha(1-r)}}{(1+2^{l}d(x,\omega;z,\mu)^{2})^{\alpha}}|\theta_{\mu,2^{-j}\sigma}(D)f(z)|^{r}\ud z\ud\mu (M_{\alpha}^{*}(f)(x,\omega,2^{-j}\sigma))^{1-r}\\
&\leq \sum_{j=1}^{\infty}2^{-|j-l|(N-\alpha)}\int_{\Sp}2^{ln}(1+2^{l}d(x,\omega;z,\mu)^{2})^{-\alpha r}|\theta_{\mu,2^{-j}\sigma}(D)f(z)|^{r}\ud z\ud\mu \big(M_{\alpha}^{*}(f)(x,\omega,2^{-j}\sigma)\big)^{1-r}.
\end{align*}
Finally, we can apply Lemmas \ref{lem1} and \ref{lem2} to this estimate to obtain
\[
(M_{\alpha}^{*}(f)(x,\omega,2^{-l}\sigma))^{r}\lesssim\sum_{j=1}^{\infty}2^{-|j-l|(N-\alpha)r}\int_{\Sp}2^{ln}(1+2^{l}d(x,\omega;z,\mu)^{2})^{-\alpha r}|\theta_{\mu,2^{-j}\sigma}(D)f(z)|^{r}\ud z\ud\mu.\qedhere
\]
\end{proof}

To conclude this subsection we collect the following result from \cite{Rychkov99}.

\begin{lemma}\label{lem3}
Let $(X,\tilde{d},\mu)$ be a metric measure space, where $\tilde{d}$ is a metric and $\mu$ is a nonnegative, doubling, Borel measure. Let $p,q\in[1,\infty)$ and $N>0$, and let $\{g_{j}\}_{j\in\Z}$  be a sequence of nonnegative measurable functions on $X$. For each $l\in\Z$ set
\begin{align*}
h_{l}:=\sum_{j=-\infty}^{\infty}2^{-|j-l|N}g_{j}.
 \end{align*}
Then there exists a $C=C(p,q,N)>0$ such that
\begin{align*}
||\{h_{l}\}_{l\in\Z}||_{L^{p}(X; \ell^{q})}\leq C ||\{g_{j}\}_{j\in\Z}||_{L^{p}(X; \ell^{q})}.
 \end{align*}
\end{lemma}

\subsection{The main embedding}

After this preliminary work, we are ready to prove the main result of this section.

\begin{theorem}\label{pro1}
One has
\[
\mathcal{H}_{FIO}^{1}(\mathbb{R}^{n})=\mathcal{H}_{FIO,G}^{1}(\mathbb{R}^{n})
\]
with equivalent norms.
\end{theorem}
\begin{proof}
As already noted, it follows from \cite[Proposition 2.1 and Remark 2.2]{AuHoMa12} that $\HT^{1}_{FIO}(\Rn)\subseteq \HT^{1}_{FIO,G}(\Rn)$ continuously. More precisely, it is shown in \cite[Proposition 2.1 and Remark 2.2]{AuHoMa12} that the following inequality holds for the square functions $S$ and $G$ from \eqref{eq:conical} and \eqref{eq:vertical}, respectively:
\begin{equation}\label{eq:squarefunctions}
\|G(g)\|_{L^{1}(\Sp)}\lesssim \|S(g)\|_{L^{1}(\Sp)}\quad(g\in\Sw'(\Rn)).
\end{equation}
So it remains to prove that $\mathcal{H}_{FIO,G}^{1}(\mathbb{R}^{n})\subseteq \mathcal{H}_{FIO}^{1}(\mathbb{R}^{n})$.

Fix $f\in \mathcal{H}_{FIO,G}^{1}(\mathbb{R}^{n})$. First note that one trivially has $q(D)f\in L^{1}(\Rn)$ and
\[
\|q(D)f\|_{L^{1}(\Rn)}\leq \|G(f)\|_{L^{1}(\Sp)}+\|q(D)f\|_{L^{1}(\Rn)}=\|f\|_{\HT^{1}_{FIO,G}(\Rn)}.
\]
So it suffices to show that $S(f)\in L^{1}(\Sp)$ and
\[
\|S(f)\|_{L^{1}(\Sp)}\lesssim \|f\|_{\HT^{1}_{FIO,G}(\Rn)}.
\]
To this end, we decompose $f$ into its low-frequency and high-frequency components:
\begin{align*}
\|S(f)\|_{L^{1}(\Sp)}\leq \|S(q(D)f)\|_{L^{1}(\Sp)}+\|S((1-q)(D)f)\|_{L^{1}(\Sp)}.
\end{align*}
For the low-frequency part we use that $\lb D\rb^{-\frac{n-1}{4}}:\HT^{1}(\Rn)\to \HT^{1}_{FIO}(\Rn)$ is continuous (see \cite[Theorem 7.4]{HaPoRo20}), where $\HT^{1}(\Rn)$ is as defined in \eqref{eq:localhardy}. Choosing $r$ in the definition of $\HT^{1}(\Rn)$ such that $r\equiv 1$ on $\supp(q)$, we obtain
\begin{equation}\label{eq:lowfreq}
\begin{aligned}
\|S(q(D)f)\|_{L^{1}(\Sp)}&\leq \|q(D)f\|_{\mathcal{H}_{FIO}^{1}(\mathbb{R}^{n})}\lesssim\|\langle D\rangle^{\frac{n-1}{4}}q(D)f\|_{\mathcal{H}^{1}(\mathbb{R}^{n})}\\
&\eqsim\|\langle D\rangle^{\frac{n-1}{4}}q(D)f\|_{L^{1}(\mathbb{R}^{n})}\lesssim\|q(D)f\|_{L^{1}(\mathbb{R}^{n})}\leq \|f\|_{\HT^{1}_{FIO,G}(\Rn)},
\end{aligned}
\end{equation}
where in the penultimate inequality we used that $q\in C^{\infty}_{c}(\Rn)$.

Next, we consider the high-frequency component $h:=(1-q)(D)f$. We fix $\alpha>n$ and claim that it suffices to prove the following two inequalities:
\begin{equation}\label{a6}
S(h)(x,\omega)\lesssim\bigg(\int_{0}^{1}[M_{\alpha}^{*}(h)(x,\omega,\sigma)]^{2}\,\frac{\ud\sigma}{\sigma}\bigg)^{1/2}
\end{equation}
for all $(x,\w)\in\Sp$, and
\begin{equation}\label{a7}
\int_{\Sp}\bigg{(}\int_{0}^{1}[M_{\alpha}^{*}(h)(x,\omega,\sigma)]^{2}\,\frac{\ud\sigma}{\sigma}\bigg{)}^{1/2}\ud x\ud\w\lesssim\|G(h)\|_{L^{1}(\Sp)}.
\end{equation}
Indeed, by combining these inequalities with \eqref{eq:lowfreq} and \eqref{eq:squarefunctions}, we obtain
\begin{align*}
\|S(f)\|_{L^{1}(\Sp)}&\leq \|S(q(D)f)\|_{L^{1}(\Sp)}+\|S(h)\|_{L^{1}(\Sp)}\lesssim \|f\|_{\HT^{1}_{FIO,G}(\Rn)}+\|G(h)\|_{L^{1}(\Sp)}\\
&\leq \|f\|_{\HT^{1}_{FIO,G}(\Rn)}+\|G(q(D)f)\|_{L^{1}(\Sp)}+\|G(f)\|_{L^{1}(\Sp)}\\
&\lesssim \|f\|_{\HT^{1}_{FIO,G}(\Rn)}+\|S(q(D)f)\|_{L^{1}(\Sp)}+\|f\|_{\HT^{1}_{FIO,G}(\Rn)}\lesssim \|f\|_{\HT^{1}_{FIO,G}(\Rn)}.
\end{align*}
Hence in the remainder we will focus on proving \eqref{a6} and \eqref{a7}.

\subsubsection*{Estimate \eqref{a6}}

This estimate follows from a straightforward calculation. For all $(x,\w)\in\Sp$, $\sigma>0$ and $(y,\nu)\in B_{\sqrt{\sigma}}(x,\omega)$ one has $1\leq 1+\sigma^{-1} d(x,\omega;y,\nu)^{2}\leq 2$. Hence
\begin{align*}
S(h)(x,\omega)&=\bigg(\int_{0}^{1}\fint_{B_{\sqrt{\sigma}}(x,\omega)}|\theta_{\nu,\sigma}(D)h(y)|^{2}\, \ud y\ud\nu\frac{\ud\sigma}{\sigma}\bigg)^{1/2}\\
&\leq \bigg(\int_{0}^{1}\sup\limits_{(y,\nu)\in B_{\sqrt{\sigma}}(x,\omega)}|\theta_{\nu,\sigma}(D)h(y)|^{2}\,\frac{\ud\sigma}{\sigma}\bigg)^{1/2}\\
&\lesssim\bigg(\int_{0}^{1}\sup\limits_{(y,\nu)\in B_{\sqrt{\sigma}}(x,\omega)}\frac{|\theta_{\nu,\sigma}(D)f(y)|^{2}}{(1+\sigma^{-1}d(x,\omega;y,\nu)^{2})^{2\alpha}}\,\frac{\ud\sigma}{\sigma}\bigg)^{1/2}\\
&\leq \bigg(\int_{0}^{1}\sup\limits_{(y,\nu)\in \Sp}\frac{|\theta_{\nu,\sigma}(D)f(y)|^{2}}{(1+\sigma^{-1}d(x,\omega;y,\nu)^{2})^{2\alpha}}\,\frac{\ud\sigma}{\sigma}\bigg)^{1/2}=\bigg(\int_{0}^{1}[M_{\alpha}^{*}(h)(x,\omega,\sigma)]^{2}\,\frac{\ud\sigma}{\sigma}\bigg)^{1/2}.
\end{align*}

\subsubsection*{Estimate \eqref{a7}}

The idea of the proof is to write
\begin{align*}
\int_{\Sp}\bigg{(}\int_{0}^{1}[M_{\alpha}^{*}(h)(x,\omega,\sigma)]^{2}\,\frac{\ud\sigma}{\sigma}\bigg{)}^{1/2}\ud x\ud\w
&=\int_{\Sp}\bigg{(}\sum_{l=1}^{\infty}\int_{1}^{2}[M_{\alpha}^{*}(h)(x,\omega,2^{-l}\sigma)]^{2}\,\frac{\ud\sigma}{\sigma}\bigg{)}^{1/2}\ud x\ud\w\\
&=\bigg\|\bigg\{\bigg(\int_{1}^{2}[M_{\alpha}^{*}(h)(\cdot,\cdot,2^{-l}\sigma)]^{2}\,\frac{\ud\sigma}{\sigma}\bigg)^{r/2}\bigg\}_{l=1}^{\infty}\bigg\|_{L^{1/r}(\Sp;\ell^{2/r})}^{1/r}
\end{align*}
for a suitably chosen $r\in(0,1)$. We will bound the sequence in the final term by a suitable expression involving the Hardy--Littlewood maximal function, and then we combine boundedness properties of this maximal function with Lemma \ref{lem3} to obtain \eqref{a7}.

For the moment, fix $(x,\w)\in\Sp$ and $l\in\N$. We will use the pointwise estimate in Proposition \ref{lem4} for $M_{\alpha}^{*}(h)$. Note that Proposition \ref{lem4} indeed applies to $h$, given that Proposition \ref{prop:Sobolev} shows that $f\in W^{-\frac{n-1}{4},1}(\Rn)$ and therefore $h=(1-q)(D)f\in W^{-\frac{n-1}{4},1}(\Rn)$ as well. And one has
\[
\F h(\xi)=(1-q(\xi))\F f(\xi)=0
\]
for $|\xi|\leq 2$ because $q(\xi)=1$ for such $\xi$. Now, since $\alpha>n$ we can choose $r\in (n/\alpha, 1)$ and $N>0$ and apply Proposition \ref{lem4} to obtain
\begin{align*}
[M_{\alpha}^{*}(h)(x,\omega,2^{-l}\sigma)]^{r}\lesssim\sum_{j=1}^{\infty}2^{-|j-l|N}\int_{\Sp}2^{ln}(1+2^{l}d(x,\omega;y,v)^{2})^{-\alpha r}|\theta_{\nu,2^{-j}\sigma}(D)h(y)|^{r}\,\ud y\ud\nu
\end{align*}
for all $l\in\N$ and $\sigma\in(1,2)$. Hence the triangle inequality and Minkowski's inequality yield
\begin{equation}\label{eq:longeq}
\begin{aligned}
&\bigg(\int_{1}^{2}[M_{\alpha}^{*}(h)(x,\omega,2^{-l}\sigma)]^{2}\,\frac{\ud\sigma}{\sigma}\bigg)^{r/2}\\
&\lesssim \bigg(\int_{1}^{2}\bigg(\sum_{j=1}^{\infty}2^{-|j-l|N}\int_{\Sp}2^{ln}(1+2^{l}d(x,\omega;y,v)^{2})^{-\alpha r}|\theta_{\nu,2^{-j}\sigma}(D)h(y)|^{r}\,\ud y\ud\nu\bigg)^{2/r}\frac{\ud\sigma}{\sigma}\bigg)^{r/2}\\
&\leq \sum_{j=1}^{\infty}2^{-|j-l|N}\int_{\Sp}2^{ln}(1+2^{l}d(x,\omega;y,v)^{2})^{-\alpha r}\bigg(\int_{1}^{2}|\theta_{\nu,2^{-j}\sigma}(D)h(y)|^{2}\frac{d\sigma}{\sigma}\bigg)^{r/2}\ud y\ud\nu.
\end{aligned}
\end{equation}
Next, we will bound each of the terms in this series separately.

Momentarily fix $j\in\N$, and write
\[
F(y,\nu):=\bigg(\int_{1}^{2}|\theta_{\nu,2^{-j}\sigma}(D)h(y)|^{2}\,\frac{\ud\sigma}{\sigma}\bigg)^{r/2}
\]
for $(y,\nu)\in\Sp$. Also let $\mathcal{M}$ be the centered Hardy-Littlewood operator on $(\Sp,d,\ud x\ud\omega)$ given by
\begin{align*}
\mathcal{M}(f)(x,\omega):=\sup_{(x,\omega)\in B}\frac{1}{V(B)}\int_{B}|f(y,\nu)| \,\ud y\ud\nu
\end{align*}
for $f\in L^{1}_{\mathrm{loc}}(\Sp)$, where the supremum is taken over all balls $B\subseteq\Sp$ with center $(x,\w)$. Then
\[
\int_{\Sp}2^{ln}\Big(1+2^{l}d(x,\omega;y,v)^{2}\Big)^{-\alpha r}F(y,\nu)\, \ud y\ud\nu=\sum_{k=0}^\infty \int_{C_{k}}\frac{1}{2^{-ln}}\Big(1+\frac{d(x,\omega;y,v)^{2}}{2^{-l}}\Big)^{-\alpha r}F(y,\nu)\,\ud y\ud\nu,
\]
where
\[
C_{0}=\{(y,\nu)\in\Sp: d(x,\omega;y,v)\leq \sqrt{2}\ 2^{-l/2}\}=B_{\sqrt{2}\,2^{-l/2}}(x,\w)
\]
and
\[
C_{k}=\{(y,\nu)\in\Sp: 2^{k/2}2^{-l/2}<d(x,\omega;y,v)\leq 2^{(k+1)/2}2^{-l/2} \}
\]
for $k\in\N$. We bound each of the terms in this series separately, recalling from Lemma \ref{lem7} that $V(B_\tau(x,\omega))\lesssim \tau^{2n}$ for all $\tau>0$. We obtain
\begin{align*}
&\int_{C_{0}}\frac{1}{2^{-ln}}\Big(1+\frac{d(x,\omega;y,v)^{2}}{2^{-l}}\Big)^{-\alpha r}F(y,\nu)\,\ud y\ud\nu\leq\int_{C_{0}}2^{ln}F(y,\nu) \, \ud y\ud\nu\\
&=V(B_{\sqrt{2}\ 2^{-l/2}}(x,\omega)) 2^{ln}\fint_{B_{\sqrt{2}\ 2^{-l/2}}(x,\omega)}F(y,\nu) \, \ud y\ud\nu\lesssim 2^n\mathcal{M}(F)(x,\omega)
\end{align*}
and, for $k\in\N$,
\begin{align*}
&\int_{C_{k}}\frac{1}{2^{-ln}}\Big(1+\frac{d(x,\omega;y,v)^{2}}{2^{-l}}\Big)^{-\alpha r}F(y,\nu)\,\ud y\ud\nu\leq \frac{1}{2^{k\alpha r}} \frac{1}{2^{-ln}}\int_{C_{k}}F(y,\nu)\, \ud y\ud\nu\\
&\leq \frac{1}{2^{k\alpha r}}\frac{1}{2^{-ln}}V(B_{2^{(k+1)/2}2^{-l/2}}(x,\omega))\fint_{B_{2^{(k+1)/2}2^{-l/2}}(x,\omega)}F(y,\nu)\, \ud y\ud\nu\lesssim \frac{1}{2^{k(\alpha r-n)}}2^n\mathcal{M}(F)(x,\omega).
\end{align*}
Since $r>\frac{n}{\alpha}$, the series converges and we obtain
\[
\int_{\Sp}2^{ln}\Big(1+2^{l}d(x,\omega;y,v)^{2}\Big)^{-\alpha r}F(y,\nu)\, \ud y\ud\nu\lesssim \mathcal{M}(F)(x,\omega)=\mathcal{M}\bigg[\bigg(\int_{1}^{2}|\theta_{\cdot,2^{-j}\sigma}(D)h(\cdot)|^{2}\,\frac{\ud\sigma}{\sigma}\bigg)^{r/2}\bigg](x,\omega).
\]
Now \eqref{eq:longeq} yields
\begin{equation}\label{eq:individualterm}
\bigg(\int_{1}^{2}[M_{\alpha}^{*}(h)(x,\omega,2^{-l}\sigma)]^{2}\,\frac{\ud\sigma}{\sigma}\bigg)^{r/2}\lesssim\sum_{j=1}^{\infty}2^{-|j-l|N}\mathcal{M}\bigg[\bigg(\int_{1}^{2}|\theta_{\cdot,2^{-j}\sigma}(D)h(\cdot)|^{2}\,\frac{\ud\sigma}{\sigma}\bigg)^{r/2}\bigg](x,\omega).
\end{equation}
We have now obtained suitable bounds for each of the terms in our original sequence, and we will use these bounds to complete the proof of \eqref{a7}.

For $(x,\w)\in\Sp$ and $j\in\N$, write
\[
g_{j}(x,\omega):=\mathcal{M}\bigg{[}\bigg{(}\int_{1}^{2}|\theta_{\cdot,2^{-j}\sigma}(D)h(\cdot)|^{2}\,\frac{\ud\sigma}{\sigma}\bigg{)}^{r/2}\bigg{]}(x,\omega),
\]
and for $l\in\N$ set
\[
h_{l}(x,\omega):=\sum_{j=1}^{\infty}2^{-|j-l|N}g_{j}(x,\omega).
\]
Then we can combine \eqref{eq:individualterm} with Lemma \ref{lem3}, as well as the boundedness of $\mathcal{M}$ on $L^{1/r}(\Sp;\ell^{2/r})$ (see \cite[Section 6.6]{GeGoKoKr98}), to obtain
\begin{align*}
&\int_{\Sp}\bigg{(}\int_{0}^{1}[M_{\alpha}^{*}(h)(x,\omega,\sigma)]^{2}\,\frac{\ud\sigma}{\sigma}\bigg{)}^{1/2}\ud x\ud\w=\bigg\|\bigg\{\bigg(\int_{1}^{2}[M_{\alpha}^{*}(h)(\cdot,\cdot,2^{-l}\sigma)]^{2}\,\frac{\ud\sigma}{\sigma}\bigg)^{r/2}\bigg\}_{l=1}^{\infty}\bigg\|_{L^{1/r}(\Sp;\ell^{2/r})}^{1/r}\\
&\lesssim \|\{h_{l}\}_{l=1}^{\infty}\|^{1/r}_{L^{1/r}(\Sp;\ell^{2/r})}\lesssim \|\{g_{j}\}_{j=1}^{\infty}\|^{1/r}_{L^{1/r}(\Sp;\ell^{2/r})}=\bigg\|\bigg\{\mathcal{M}\bigg[\bigg(\int_{1}^{2}|\theta_{\cdot,2^{-j}\sigma}(D)h(\cdot)|^{2}\,\frac{\ud\sigma}{\sigma}\bigg)^{r/2}\bigg]\bigg\}_{j=1}^{\infty}\bigg\|_{L^{1/r}(\Sp;\ell^{2/r})}^{1/r}\\
&\lesssim\bigg\|\bigg\{\bigg(\int_{1}^{2}|\theta_{\cdot,2^{-j}\sigma}(D)h(\cdot)|^{2}\,\frac{\ud\sigma}{\sigma}\bigg)^{r/2}\bigg\}_{j=1}^{\infty}\bigg\|_{L^{1/r}(\Sp;\ell^{2/r})}^{1/r}=\bigg\|\bigg(\int_{0}^{1}|\theta_{\cdot,\sigma}(D)h(\cdot)|^{2}\,\frac{\ud\sigma}{\sigma}\bigg)^{1/2}\bigg\|_{L^{1}(\Sp)}\\
&=\|G(h)\|_{L^{1}(\Sp)}.
\end{align*}
This concludes the proof of \eqref{a7} and thereby of the theorem.
\end{proof}

\section{Maximal function characterization}\label{sec:maximal}

In \cite{Rozendaal21} a maximal function characterization of $\mathcal{H}_{FIO}^{p}(\mathbb{R}^{n})$ was obtained for $1<p<\infty$. As an immediate corollary of what we have already shown, we can extend this characterization to $\mathcal{H}_{FIO}^{1}(\mathbb{R}^{n})$, by showing that $\HT^{1}_{FIO}(\Rn)=\HT^{1}_{FIO,\max}(\Rn)$.

\begin{theorem}\label{pro3}
One has
\[
\mathcal{H}_{FIO}^{1}(\mathbb{R}^{n})= \mathcal{H}_{FIO,\max}^{1}(\mathbb{R}^{n})
\]
with equivalence of norms.
\end{theorem}

\begin{proof}
From the maximal function characterization of $H^{1}(\Rn)$ (see \cite[Theorem 2.1.4]{Grafakos14b}), we know that an $f\in\Sw'(\Rn)$ satisfies $f\in \mathcal{H}_{FIO,\max}^{1}(\mathbb{R}^{n})$ if and only if $q(D)f\in L^{1}(\Rn)$, $\ph_{\w}(D)f\in H^{1}(\Rn)$ for almost all $\w\in S^{n-1}$, and
\[
\int_{S^{n-1}}\|\varphi_{\omega}(D)f\|_{H^{1}(\mathbb{R}^{n})}\,\ud \omega<\infty.
\]
Moreover, in this case one has
\[
\|f\|_{\mathcal{H}_{FIO,\max}^{1}(\mathbb{R}^{n})}\eqsim \int_{S^{n-1}}\|\varphi_{\omega}(D)f\|_{H^{1}(\mathbb{R}^{n})}\,\ud \omega+\|q(D)f\|_{L^{1}(\mathbb{R}^{n})}.
\]
Hence the required statement follows from Proposition \ref{prop:equivalent} and Theorem \ref{pro1}.
\end{proof}

\begin{remark}\label{rem:pinfty}
The proof of Theorem \ref{pro3} relies on the following characterization of $\HT^{1}_{FIO}(\Rn)$, obtained by combining Proposition \ref{prop:equivalent} and Theorem \ref{pro1}: an $f\in\Sw'(\Rn)$ satisfies $f\in \HT^{1}_{FIO}(\Rn)$ if and only if $q(D)f\in L^{1}(\Rn)$, $\ph_{\w}(D)f\in H^{1}(\Rn)$ for almost all $\w\in S^{n-1}$, and $\int_{S^{n-1}}\|\varphi_{\omega}(D)f\|_{H^{1}(\mathbb{R}^{n})}\,\ud \omega<\infty$. In this case one has
\[
\|f\|_{\mathcal{H}_{FIO}^{1}(\mathbb{R}^{n})}\eqsim \int_{S^{n-1}}\|\varphi_{\omega}(D)f\|_{H^{1}(\mathbb{R}^{n})}\,\ud \omega+\|q(D)f\|_{L^{1}(\mathbb{R}^{n})}.
\]
A similar characterization of $\Hp$ was obtained in \cite{Rozendaal21} for $1<p<\infty$, but it is not clear whether one can also characterize $\HT^{\infty}_{FIO}(\Rn)$ in this manner. More precisely, a natural question is whether an $f\in\Sw'(\Rn)$ satisfies $f\in \HT^{\infty}_{FIO}(\Rn)$ if and only if $q(D)f\in L^{\infty}(\Rn)$, $\ph_{\w}(D)f\in \BMO(\Rn)$ for almost all $\w\in S^{n-1}$, and $\esssup_{\w\in S^{n-1}}\|\varphi_{\omega}(D)f\|_{\BMO(\mathbb{R}^{n})}<\infty$, and whether in this case
\[
\|f\|_{\mathcal{H}_{FIO}^{\infty}(\mathbb{R}^{n})}\eqsim \esssup_{\w\in S^{n-1}}\|\varphi_{\omega}(D)f\|_{\BMO(\mathbb{R}^{n})}+\|q(D)f\|_{L^{\infty}(\mathbb{R}^{n})}.
\]
One can use duality to show that if $q(D)f\in L^{\infty}(\Rn)$, $\ph_{\w}(D)f\in \BMO(\Rn)$ for almost all $\w\in S^{n-1}$, and $\esssup_{\w\in S^{n-1}}\|\varphi_{\omega}(D)f\|_{\BMO(\mathbb{R}^{n})}<\infty$, then $f\in\HT^{\infty}_{FIO}(\Rn)$ with
\[
\|f\|_{\mathcal{H}_{FIO}^{\infty}(\mathbb{R}^{n})}\lesssim \esssup_{\w\in S^{n-1}}\|\varphi_{\omega}(D)f\|_{\BMO(\mathbb{R}^{n})}+\|q(D)f\|_{L^{\infty}(\mathbb{R}^{n})}.
\]
However, it is not clear whether the reverse inequality also holds. We leave this as an open problem.
\end{remark}

\section{${\mathcal G}_{\alpha}^{*}$ characterization}\label{sec:conical}

In this section, we will prove that $\mathcal{H}_{FIO}^{1}(\mathbb{R}^{n})= \mathcal{H}_{FIO,\mathcal{G}_{\alpha}^{*}}^{1}(\mathbb{R}^{n})$ for $\alpha>2$. To do so, we will need the following quantitative change of aperture formula from \cite[Lemma 2.2]{Rozendaal21} (see also \cite{Auscher11}).

\begin{lemma}\label{lem5}
There exists a $C\geq 0$ such that, for all $\lambda\geq 1$ and $F\in L^{2}_{\mathrm{loc}}(\Sp\times(0,\infty))$, one has
\begin{align*}
&\int_{\Sp}\bigg(\int_{0}^{\infty}\fint_{B_{\lambda\sqrt{\sigma}}(x,\omega)}|F(y,\nu,\sigma)|^{2}\,\ud y\ud \nu \frac{\ud \sigma}{\sigma}\bigg)^{1/2}\ud x\ud \omega\\
&\leq C \lambda^{n}\int_{\Sp}\bigg(\int_{0}^{\infty}\fint_{B_{\sqrt{\sigma}}(x,\omega)}|F(y,\nu,\sigma)|^{2}\,\ud y\ud \nu \frac{\ud \sigma}{\sigma}\bigg)^{1/2}\ud x\ud \omega
\end{align*}
whenever the second term is finite.
\end{lemma}

For the next theorem, recall that $\mathcal{H}_{FIO,\mathcal{G}_{\alpha}^{*}}^{1}(\mathbb{R}^{n})$ consists of all $f\in \Sw'(\mathbb{R}^{n})$ such that $\mathcal{G}_{\alpha}^{*}(f)\in L^{1}(\Sp)$ and $q(D)f\in L^{1}(\mathbb{R}^{n})$, endowed with the norm
\[
\|f\|_{\mathcal{H}_{FIO,\mathcal{G}_{\alpha}^{*}}^{1}(\mathbb{R}^{n})}=\|\mathcal{G}_{\alpha}^{*}(f)\|_{L^{1}(\Sp)}+\|q(D)f\|_{L^{1}(\mathbb{R}^{n})}.
\]
Here
\[
\mathcal{G}_{\alpha}^{*}(f)(x,\omega)=\bigg(\int_{0}^{1}\int_{\Sp}
\frac{|\theta_{\nu,\sigma}(D)f(y)|^{2}}{\sigma^{n}(1+\sigma^{-1}
d(x,\omega;y,\nu)^{2})^{n\alpha}}\, \ud y\ud\nu\frac{\ud\sigma}{\sigma}\bigg)^{1/2}
\]
for $(x,\w)\in\Sp$.

\begin{theorem}\label{pro2}
Let $\alpha>2$. Then
\[
\mathcal{H}_{FIO}^{1}(\mathbb{R}^{n})= \mathcal{H}_{FIO,\mathcal{G}_{\alpha}^{*}}^{1}(\mathbb{R}^{n})
\]
with equivalent norms.
\end{theorem}
\begin{proof}
We first show that $\mathcal{H}_{FIO,\mathcal{G}_{\alpha}^{*}}^{1}(\mathbb{R}^{n})\subseteq\HT^{1}_{FIO}(\Rn)$. Let $f\in \mathcal{H}_{FIO,\mathcal{G}_{\alpha}^{*}}^{1}(\mathbb{R}^{n})$. It suffices to prove that $S(f)\in L^{1}(\Sp)$ with $\|S(f)\|_{L^{1}(\Sp)}\lesssim \|\mathcal{G}_{\alpha}^{*}(f)\|_{L^{1}(\Sp)}$. To this end, observe that for all $(x,\w)\in\Sp$, $\sigma>0$ and $(y,\nu)\in B_{\sqrt{\sigma}}(x,\w)$, one has $1+\sigma^{-1}d(x,\omega;y,v)^{2}\leq 2$. Moreover, Lemma \ref{lem7} yields that $V(B_{\sqrt{\sigma}}(x,\w))\eqsim \sigma^{n}$ for all $\sigma\in(0,1)$. Hence
\begin{align*}
S(f)(x,\omega)&=\bigg(\int_{0}^{1}\fint_{B_{\sqrt{\sigma}}(x,\omega)}|\theta_{\nu,\sigma}(D)f(y)|^{2} \,\ud y\ud \nu\frac{\ud \sigma}{\sigma}\bigg)^{1/2}\\
&\lesssim \bigg(\int_{0}^{1}\int_{B_{\sqrt{\sigma}}(x,\omega)}
\frac{|\theta_{\nu,\sigma}(D)f(y)|^{2}}{V(B_{\sqrt{\sigma}}(x,\w))(1+\sigma^{-1}d(x,\omega;y,\nu)^{2})^{n\alpha}}\,\ud y\ud \nu\frac{\ud \sigma}{\sigma}\bigg)^{1/2}\\
&\lesssim \mathcal{G}_{\alpha}^{*}(f)(x,\omega).
\end{align*}
Thus $S(f)\in L^{1}(\Sp)$ with $\|S(f)\|_{L^{1}(\Sp)}\lesssim \|\mathcal{G}_{\alpha}^{*}(f)\|_{L^{1}(\Sp)}$.

For the other inclusion we let $f\in \HT^{1}_{FIO}(\Rn)$ and show that $\|\mathcal{G}_{\alpha}^{*}(f)\|_{L^{1}(\Sp)}\lesssim \|S(f)\|_{L^{1}(\Sp)}$. Note that
\begin{align*}
(1+s)^{-n\alpha}\leq\ind_{[0,1]}(s)+2 \sum_{k=1}^{\infty}2^{-k}\mathbf{1}_{_{[0,1]}}\bigg(\frac{s}{2^{k/(n\alpha)}}\bigg)=:g(s)
\end{align*}
for all $s\geq0$, as can be seen for $s>1$ by letting $k_{0}\in\N$ be such that $2^{(k_{0}-1)/n\alpha}<s\leq 2^{k_{0}/n\alpha}$. Now apply Lemma \ref{lem7} to obtain, for all $(x,\w)\in\Sp$,
\begin{align*}
&\mathcal{G}_{\alpha}^{*}(f)(x,\omega)
\leq \bigg(\int_{0}^{1}\int_{\Sp}\sigma^{-n}|\theta_{\nu,\sigma}(D)f(y)|^{2}g(\sigma^{-1}d(x,\omega;y,\nu)^{2}) \,\ud y\ud \nu\frac{\ud \sigma}{\sigma}\bigg)^{\frac{1}{2}}\\
&\eqsim \bigg(\int_{0}^{1}\int_{B_{\sqrt{\sigma}}(x,\omega)}\sigma^{-n}|\theta_{\nu,\sigma}(D)f(y)|^{2}
\ud y\ud \nu\frac{\ud \sigma}{\sigma}+\sum_{k=1}^{\infty}2^{-k}\int_{0}^{1}\int_{B_{2^{k/(2n\alpha)}\sqrt{\sigma}}(x,\omega)}\sigma^{-n}|\theta_{\nu,\sigma}(D)f(y)|^{2}
\,\ud y\ud \nu\frac{\ud \sigma}{\sigma}\bigg)^{\frac{1}{2}}\\
&\lesssim \bigg(\int_{0}^{1}\fint_{B_{\sqrt{\sigma}}(x,\omega)}|\theta_{\nu,\sigma}(D)f(y)|^{2}
\ud y\ud \nu\frac{\ud \sigma}{\sigma}+\sum_{k=1}^{\infty}2^{-k+k/\alpha}\int_{0}^{1}\fint_{B_{2^{k/(2n\alpha)}\sqrt{\sigma}}(x,\omega)}|\theta_{\nu,\sigma}(D)f(y)|^{2}
\ud y\ud \nu\frac{\ud \sigma}{\sigma}\bigg)^{\frac{1}{2}}\\
&\lesssim \bigg(\int_{0}^{1}\fint_{B_{\sqrt{\sigma}}(x,\omega)}|\theta_{\nu,\sigma}(D)f(y)|^{2}
\ud y\ud \nu\frac{\ud \sigma}{\sigma}\bigg)^{\frac{1}{2}}+\sum_{k=1}^{\infty}2^{(-k+k/\alpha)/2}\bigg(\int_{0}^{1}\fint_{B_{2^{k/(2n\alpha)}\sqrt{\sigma}}(x,\omega)}|\theta_{\nu,\sigma}(D)f(y)|^{2}
\ud y\ud \nu\frac{\ud \sigma}{\sigma}\bigg)^{\frac{1}{2}}.
\end{align*}
We can then conclude the proof using Lemma \ref{lem5}, with $F(y,\nu,\sigma)=\theta_{\nu,\sigma}(D)f(y)$:
\begin{align*}
&\|\mathcal{G}_{\alpha}^{*}(f)\|_{L^{1}(\Sp)}\\
&\lesssim \|S(f)\|_{L^{1}(\Sp)}+\sum_{k=1}^{\infty}2^{-k/2+k/(2\alpha)}\int_{\Sp}\bigg(\int_{0}^{1}\fint_{B_{2^{k/(2n\alpha)}\sqrt{\sigma}}(x,\omega)}|\theta_{\nu,\sigma}(D)f(y)|^{2}
\ud y\ud \nu\frac{\ud \sigma}{\sigma}\bigg)^{1/2}\ud x\ud \omega\\
&\lesssim \|S(f)\|_{L^{1}(\Sp)}+\sum_{k=1}^{\infty}2^{-k/2+k/\alpha}\|S(f)\|_{L^{1}(\Sp)}\lesssim \|S(f)\|_{L^{1}(\Sp)},
\end{align*}
where for the final inequality we used that $\alpha>2$.
\end{proof}

\section{Applications}\label{sec:applications}

In this section we give two applications of the results in the previous sections. The aim here is to show how the characterizations in this article can be used to incorporate techniques from other parts of harmonic analyis, and to demonstrate that the characterizations are amenable to direct calculations. Other applications of these characterizations, to operators with rough coefficients, will follow in future work.

We first prove that a large class of singular integral operators which are bounded on $L^{p}(\Rn)$ for $1<p<\infty$ are also bounded on $\Hp$ for all $1\leq p\leq\infty$. Recall the definition of the local Hardy space $\HT^{1}(\Rn)$ from \eqref{eq:localhardy}.

\begin{theorem}\label{thm:czo}
Let $m\in L^{\infty}(\Rn)$ be such that $m(D):\HT^{1}(\Rn)\to L^{1}(\Rn)$ is bounded. Then $m(D):\Hp\to\Hp$ is bounded for all $p\in[1,\infty]$.
\end{theorem}
\begin{proof}
We first consider the case where $p=1$. It follows from the inclusion $H^{1}(\Rn)\subseteq L^{1}(\Rn)$ that $H^{1}(\Rn)\subseteq\HT^{1}(\Rn)$, and therefore $m(D):H^{1}(\Rn)\to L^{1}(\Rn)$ is bounded. Now, for $j\in\{1,\ldots, n\}$, let $R_{j}(D)$, where $R_{j}(\xi):=-i\xi_{j}/|\xi|$ for $\xi=(\xi_{1},\ldots, \xi_{n})\in\Rn\setminus \{0\}$, be the $j$-th Riesz transform. Then the Riesz transform characterization of $H^{1}(\Rn)$ (see \cite[Section III.4.3]{Stein93}) shows that $m(D):H^{1}(\Rn)\to H^{1}(\Rn)$ with
\begin{align*}
\|m(D)f\|_{H^{1}(\Rn)}&\eqsim \|m(D)f\|_{L^{1}(\Rn)}+\sum_{j=1}^{n}\|R_{j}(D)m(D)f\|_{L^{1}(\Rn)}\lesssim \|f\|_{H^{1}(\Rn)}+\sum_{j=1}^{n}\|m(D)R_{j}(D)f\|_{L^{1}(\Rn)}\\
&\lesssim \|f\|_{H^{1}(\Rn)}+\sum_{j=1}^{n}\|R_{j}(D)f\|_{H^{1}(\Rn)}\lesssim \|f\|_{H^{1}(\Rn)}
\end{align*}
for all $f\in H^{1}(\Rn)$, where we also used that the Riesz transforms are bounded on $H^{1}(\Rn)$.

Now let $f\in\HT^{1}_{FIO}(\Rn)$. By Proposition \ref{prop:equivalent} and Theorem \ref{pro1}, it suffices to show that $q(D)m(D)f\in L^{1}(\Rn)$, $\ph_{\w}(D)m(D)f\in H^{1}(\Rn)$ for almost all $\w\in S^{n-1}$, and
\[
\int_{S^{n-1}}\|\varphi_{\omega}(D)m(D)f\|_{H^{1}(\mathbb{R}^{n})}\,\ud \omega+\|q(D)m(D)f\|_{L^{1}(\Rn)}\lesssim \|f\|_{\HT^{1}_{FIO}(\Rn)}.
\]
But this follows from the boundedness of $m(D)$ on $H^{1}(\Rn)$ and from $\HT^{1}(\Rn)$ to $L^{1}(\Rn)$, if one takes $r\in C^{\infty}_{c}(\Rn)$ in the definition of $\HT^{1}(\Rn)$ such that $r\equiv 1$ on $\supp(q)$:
\begin{align*}
&\int_{S^{n-1}}\|\varphi_{\omega}(D)m(D)f\|_{H^{1}(\mathbb{R}^{n})}\,\ud \omega+\|q(D)m(D)f\|_{L^{1}(\Rn)}\\
&=\int_{S^{n-1}}\|m(D)\varphi_{\omega}(D)f\|_{H^{1}(\mathbb{R}^{n})}\,\ud \omega+\|m(D)q(D)f\|_{L^{1}(\Rn)}\lesssim \int_{S^{n-1}}\|\varphi_{\omega}(D)f\|_{H^{1}(\mathbb{R}^{n})}\,\ud \omega+\|q(D)f\|_{\HT^{1}(\Rn)}\\
&=\int_{S^{n-1}}\|\varphi_{\omega}(D)f\|_{H^{1}(\mathbb{R}^{n})}\,\ud \omega+\|q(D)f\|_{L^{1}(\Rn)}\eqsim \|f\|_{\HT^{1}_{FIO}(\Rn)},
\end{align*}
where we again used Proposition \ref{prop:equivalent} and Theorem \ref{pro1} for the final equivalence.

Next, we consider the case $p=\infty$. It is straightforward to check that
\[
\overline{m(D)^{*}g(x)}=\overline{\overline{m}(D)g(x)}=m(D)\tilde{g}(-x)
\]
for all $g\in\Sw(\Rn)$  and $x\in\Rn$, where $\tilde{g}(y):=\overline{g(-y)}$ for $y\in\Rn$. Hence
\[
\|m(D)^{*}g\|_{L^{1}(\Rn)}=\big\|\overline{m(D)^{*}g}\big\|_{L^{1}(\Rn)}=\|m(D)\tilde{g}\|_{L^{1}(\Rn)}\lesssim \|\tilde{g}\|_{\HT^{1}(\Rn)}=\|g\|_{\HT^{1}(\Rn)},
\]
so that $m(D)^{*}:\HT^{1}(\Rn)\to L^{1}(\Rn)$. It now follows from what we have shown for $p=1$ that $m(D)^{*}:\HT^{1}_{FIO}(\Rn)\to \HT^{1}_{FIO}(\Rn)$ is continuous. Hence \eqref{eq:Hinfty} and the density of $\Sw(\Rn)$ in $\HT^{1}_{FIO}(\Rn)$ (see \cite[Proposition 6.6]{HaPoRo20}) imply that for all $f\in\HT^{\infty}_{FIO}(\Rn)$ one has
\begin{align*}
\|m(D)f\|_{\HT^{\infty}_{FIO}(\Rn)}&\eqsim \sup|\lb m(D)f,g\rb_{\Rn}|=\sup |\lb f,m(D)^{*}g\rb_{\Rn}|\lesssim \sup \|f\|_{\HT^{\infty}_{FIO}(\Rn)}\|m(D)^{*}g\|_{\HT^{1}_{FIO}(\Rn)}\\
&\lesssim \sup \|f\|_{\HT^{\infty}_{FIO}(\Rn)}\|g\|_{\HT^{1}_{FIO}(\Rn)}=\|f\|_{\HT^{\infty}_{FIO}(\Rn)},
\end{align*}
where the supremum is taken over all $g\in\Sw(\Rn)$ such that $\|g\|_{\HT^{1}_{FIO}(\Rn)}\leq 1$. This proves the required statement for $p=\infty$.

Finally, for $1<p<\infty$ one can use complex interpolation, by \cite[Proposition 6.7]{HaPoRo20}.
\end{proof}

\begin{remark}\label{rem:boundedness}
For sufficiently smooth $m$ the conclusion of Theorem \ref{thm:czo} was already obtained in \cite[Theorem 6.10]{HaPoRo20}. This is the case, for example, if $m\in C^{\infty}(\Rn)$ satisfies standard symbol estimates of the form
\begin{equation}\label{eq:S0}
|\partial^{\alpha}_{\xi}m(\xi)|\leq C_{\alpha} \lb \xi\rb^{-|\alpha|}\quad(\xi\in\Rn)
\end{equation}
for all $\alpha\in\Z_{+}^{n}$, and such estimates hold e.g.~for the local Riesz transforms. However, the techniques used to prove \cite[Theorem 6.10]{HaPoRo20} involve repeated integration by parts and require more regularity than the Mikhlin multiplier theorem. Hence Theorem \ref{thm:czo} allows one to incorporate results from other parts of harmonic analysis that are not accessible without the characterizations in this article.
\end{remark}

In \cite[Theorem 6.10]{HaPoRo20} it is shown that $m(D):\Hp\to\Hp$ for all $1\leq p\leq\infty$ if $m\in C^{\infty}(\Rn)$ is such that for all $\alpha\in\Z_{+}^{n}$ and $\beta\in\Z_{+}$ there exists a $C_{\alpha,\beta}\geq0$ with
\[
|\lb \hat{\xi},\nabla_{\xi}\rb^{\beta}\partial_{\xi}^{\alpha}m(\xi)|\leq C_{\alpha,\beta}\lb \xi\rb^{-\frac{|\alpha|}{2}-\beta}\quad(\xi\in\Rn\setminus\{0\}).
\]
It is not clear whether one also has $m(D):\Hp\to\Hp$ for some $p\neq 2$ under the weaker assumption that $m\in S^{0}_{1/2}(\Rn)$. Here $S^{\gamma}_{1/2}(\Rn)$, for $\gamma\in\R$, consists of all $m\in C^{\infty}(\Rn)$ such that for all $\alpha\in\Z_{+}^{n}$ there exists a $C_{\alpha}\geq0$ with
\[
|\partial_{\xi}^{\alpha}m(\xi)|\leq C_{\alpha,\beta}\lb \xi\rb^{\gamma-\frac{|\alpha|}{2}}\quad(\xi\in\Rn).
\]
Using the alternative characterizations of $\Hp$ we can easily obtain a slightly weaker result.

\begin{corollary}\label{cor:Shalf}
Let $\gamma\in[0,n/4]$. Then each $m\in S^{-\gamma}_{1/2}(\Rn)$ satisfies $m(D):\Hp\to\Hp$ for all $p\in[1,\infty]$ with $\big|\frac{1}{2}-\frac{1}{p}\big|\leq 2\gamma/n$.
\end{corollary}
\begin{proof}
For $\gamma=n/4$ one has $m(D):\HT^{1}(\Rn)\to L^{1}(\Rn)$ by \cite[Section VII.5.12]{Stein93}, and then Theorem \ref{thm:czo} concludes the proof. For $\gamma=0$ the result follows from Plancherel's theorem, given that $\HT^{2}_{FIO}(\Rn)=L^{2}(\Rn)$. Stein interpolation then yields the required result for $0<\gamma<n/4$. Alternatively, for $0< \gamma<n/4$ one can directly combine the characterization of $\Hp$ from \eqref{eq:characterizationp} with $L^{p}(\Rn)$-bounds for $m(D)$ from \cite[Section VII.5.12]{Stein93}.
\end{proof}

Next, we determine in a relatively explicit manner the $\Hp$ norm, for $1\leq p\leq \infty$, of functions with frequency support in one of the dyadic-parabolic regions in \eqref{eq:dyadicpar}. For simplicity of notation we write $\HT^{p}(\Rn)=L^{p}(\Rn)$ for $1<p<\infty$, and $\HT^{\infty}(\Rn)=\text{bmo}(\Rn)=(\HT^{1}(\Rn))^{*}$.

\begin{proposition}\label{prop:examples}
Let $p\in[1,\infty]$ and set $s_{p}:=\frac{n-1}{2}\big|\frac{1}{2}-\frac{1}{p}\big|$. Then for each $A>1$ there exists a $C>0$ such that the following statements hold for all $f\in\HT^{p}_{FIO}(\Rn)$. Suppose that there exist $\tau>0$ and $\nu\in S^{n-1}$ with
\begin{equation}\label{eq:supportwp}
\supp(\F(f))\subseteq\{\xi\in\Rn\mid |\xi|\in [\tau^{-1}/A,A\tau^{-1}], |\hat{\xi}-\nu|\leq A\sqrt{\tau}\}.
\end{equation}
Then the following assertions hold.
\begin{enumerate}
\item\label{it:sobeq1}
If $p\leq 2$, then
\begin{equation}\label{eq:sobeq1}
\frac{1}{C}\|\lb D\rb^{-s_{p}}f\|_{\HT^{p}(\Rn)} \leq \|f\|_{\HT^{p}_{FIO}(\Rn)}\leq C \|\lb D\rb^{-s_{p}}f\|_{\HT^{p}(\Rn)}.
\end{equation}
Hence there does not exist an $s<s_{p}$ such that $\lb D\rb^{-s}:\Hp\to \HT^{p}(\Rn)$ is bounded.
\item\label{it:sobeq2}
If $p>2$, then
\begin{equation}\label{eq:sobeq2}
\frac{1}{C}\|\lb D\rb^{s_{p}}f\|_{\HT^{p}(\Rn)} \leq \|f\|_{\HT^{p}_{FIO}(\Rn)}\leq C \|\lb D\rb^{s_{p}}f\|_{\HT^{p}(\Rn)}.
\end{equation}
Hence there does not exist an $s<s_{p}$ such that $\lb D\rb^{-s}:\HT^{p}(\Rn)\to\Hp$ is bounded.
\end{enumerate}
\end{proposition}
Note that \eqref{eq:supportwp} holds in particular for the wave packets $\F^{-1}(\theta_{\nu,\tau})$ and $\F^{-1}(\chi_{\nu,\tau})$, by Lemma \ref{WP}.
\begin{proof}
We first deal with the low frequencies of $f$. Let $r,r'\in C^{\infty}_{c}(\Rn)$ be such that $r\equiv 1$ on $\supp(q)$, and $r'\equiv 1$ on $\supp(r)$. Then $\lb D\rb^{s}r'(D):\HT^{p}(\Rn)\to\HT^{p}(\Rn)$ is bounded for all $s\in\R$, so the Sobolev embeddings for $\HT^{p}_{FIO}(\Rn)$ from \cite[Theorem 7.4]{HaPoRo20} yield
\begin{align*}
\|\lb D\rb^{s_{p}}r(D)f\|_{\HT^{p}(\Rn)}&=\|\lb D\rb^{2s_{p}}r'(D)\lb D\rb^{-s_{p}}r(D)f\|_{\HT^{p}(\Rn)}\lesssim \|\lb D\rb^{-s_{p}}r(D)f\|_{\HT^{p}(\Rn)}\\
&\lesssim \|r(D)f\|_{\HT^{p}_{FIO}(\Rn)}\lesssim \|\lb D\rb^{s_{p}}r(D)f\|_{\HT^{p}(\Rn)}\\
&=\|\lb D\rb^{2s_{p}}r'(D)\lb D\rb^{-s_{p}}r(D)f\|_{\HT^{1}(\Rn)}\lesssim \|\lb D\rb^{-s_{p}}r(D)f\|_{\HT^{p}(\Rn)}.
\end{align*}
Hence all the norms of $r(D)f$ under consideration are equivalent, and it suffices to prove \eqref{eq:sobeq1} and \eqref{eq:sobeq2} with $f$ replaced by $g:=(1-r)(D)f$. Note that $q(D)g=0$.

\eqref{it:sobeq1}: By the Sobolev embeddings for $\HT^{p}_{FIO}(\Rn)$ one has $\|\lb D\rb^{-s_{p}}g\|_{\HT^{p}(\Rn)}\lesssim \|g\|_{\HT^{p}_{FIO}(\Rn)}$, so it remains to show that $\|g\|_{\HT^{p}_{FIO}(\Rn)}\lesssim \|\lb D\rb^{-s_{p}}g\|_{\HT^{p}(\Rn)}$. To this end, first note that $\ph_{\w}(\xi)=0$ if $|\xi|<1/8$ or $|\hat{\xi}-\w|>2|\xi|^{-1/2}$ (see e.g.~\cite[Remark 3.3]{Rozendaal21}). It is then easy to check, using the support properties of $\F(g)$, that $\theta_{\w,\sigma}(D)g=0$ if $|\w-\nu|>3A\sqrt{\tau}$ or $\sigma\notin[\tau/(2A),2A\tau]$. Let $E_{\nu,\tau}:=\{\w\in S^{n-1}\mid |\w-\nu|\leq 3A\sqrt{\tau}\}$, and note that $|E_{\nu,\tau}|\eqsim \tau^{\frac{n-1}{2}}$ for implicit constants independent of $\nu$ and $\tau$. We now use the characterization in Theorem \ref{pro1} for $p=1$, and the corresponding one in \cite[Corollary 4.5]{Rozendaal21} for $1<p\leq 2$. By combining this with the bounds for $\F^{-1}(\theta_{\w,\sigma})$ from \eqref{fourier}, we obtain
\begin{align*}
\|g\|_{\HT^{p}_{FIO}(\Rn)}^{p}&\eqsim \int_{\Sp}\bigg(\int_{0}^{1}|\theta_{\w,\sigma}(D)g(x)|^{2}\frac{\ud\sigma}{\sigma}\bigg)^{p/2}\ud x\ud\w= \int_{\Rn}\int_{E_{\nu,\tau}}\bigg(\int_{\tau/4}^{4\tau}|\theta_{\w,\sigma}(D)g(x)|^{2}\frac{\ud\sigma}{\sigma}\bigg)^{p/2}
\ud \w\ud x\\
&\lesssim \int_{E_{\nu,\tau}}\int_{\Rn}\sup_{\sigma\in [\tau/(2A), 2A\tau]}|\theta_{\w,\sigma}(D)g(x)|^{p}\ud x\ud \w\\
&\lesssim \int_{E_{\nu,\tau}}\int_{\Rn}\sup_{\sigma\in [\tau/(2A), 2A\tau]}\bigg(\int_{\Rn}\sigma^{-\frac{3n+1}{4}}(1+\sigma^{-1}|x-y|^{2}+\sigma^{-2}\lb \w,x-y\rb^{2})^{-(n+1)}|g(y)|\ud y\bigg)^{p}\ud x\ud \w\\
&\eqsim \tau^{-p\frac{n-1}{4}}\int_{E_{\nu,\tau}}\int_{\Rn}\bigg(\int_{\Rn}\tau^{-\frac{n+1}{2}}(1+\tau^{-1}|x-y|^{2}+\tau^{-2}\lb \w,x-y\rb^{2})^{-(n+1)}|g(y)|\ud y\bigg)^{p}\ud x\ud \w.\\
\end{align*}
Now, an anisotropic substitution shows that $
\int_{\Rn}\tau^{-\frac{n+1}{2}}(1+\tau^{-1}|z|^{2}+\tau^{-2}\lb \w,z\rb^{2})^{-(n+1)}\ud z\lesssim 1$ for all $\w\in {S}^{n-1}$. Using this twice, in conjunction with H\"{o}lder's inequality, we obtain
\begin{align*}
&\int_{\Rn}\bigg(\int_{\Rn}\tau^{-\frac{n+1}{2}}(1+\tau^{-1}|x-y|^{2}+\tau^{-2}\lb \w,x-y\rb^{2})^{-(n+1)}|g(y)|\ud y\bigg)^{p} \ud x\\
&\lesssim \int_{\Rn}\int_{\Rn}\tau^{-\frac{n+1}{2}}(1+\tau^{-1}|x-y|^{2}+\tau^{-2}\lb \w,x-y\rb^{2})^{-(n+1)}|g(y)|^{p}\ud y \ud x\lesssim \|g\|_{L^{p}(\Rn)}^{p}
\end{align*}
for each $\w\in {S}^{n-1}$. It follows that
\begin{align*}
\|g\|_{\Hp}^{p}\lesssim \tau^{-p\frac{n-1}{4}}\int_{E_{\nu,\tau}} \|g\|_{L^{p}(\Rn)}^{p} \ud \w
\eqsim \tau^{ps_{p}}\|g\|_{L^{p}(\Rn)}^{p} \eqsim \|\lb D\rb^{-s_{p}}g\|_{\HT^{p}(\Rn)}^{p}.
\end{align*}
The very last equivalence of norms is derived in a standard manner from the support properties of $\F(f)$, using for example a Littlewood--Paley description of the $\HT^{p}(\Rn)$-norm and a change of square functions. This proves \eqref{eq:sobeq1}.

To conclude the proof of \eqref{it:sobeq1}, we will apply \eqref{eq:sobeq1} to $\F^{-1}(\theta_{\nu,\tau})$ for $\tau\in(0,1)$ and a given $\nu\in S^{n-1}$. Let $\Psi'\in C^{\infty}_{c}(\Rn)$ be such that $\Psi'(\xi)=0$ for $|\xi|\notin [1/4,4]$ and such that $\Psi'\equiv1$ on $\supp(\Psi)$. Then for all $s\in\R$ and $\xi\in\supp(\theta_{\nu,\tau})$ one has
\begin{align*}
\lb \xi\rb^{-s_{p}}\theta_{\nu,\tau}(\xi)=\tau^{s_{p}-s}m_{\tau}(\xi)\lb\xi\rb^{-s}\theta_{\nu,\tau}(\xi),
\end{align*}
where $m_{\tau}(\xi):=\lb \xi\rb^{s-s_{p}}\tau^{s-s_{p}}\Psi'(\tau\xi)$ for $\xi\neq0$. Note that $m_{\tau}$ satisfies standard symbol estimates as in \eqref{eq:S0}, with constants independent of $\tau$. Hence $m_{\tau}(D):\HT^{p}(\Rn)\to\HT^{p}(\Rn)$ is bounded, uniformly in $\tau\in(0,1)$. Then, by \eqref{eq:sobeq1},
\begin{align*}
\|\F^{-1}(\theta_{\nu,\tau})\|_{\Hp}&\eqsim \|\lb D\rb^{-s_{p}}\F^{-1}(\theta_{\nu,\tau})\|_{\HT^{p}(\Rn)}=\tau^{s_{p}-s}\|m_{\tau}(D)\lb D\rb^{-s}\F^{-1}(\theta_{\nu,\tau})\|_{\HT^{p}(\Rn)}\\
&\lesssim \tau^{s_{p}-s}\|\lb D\rb^{-s}\F^{-1}(\theta_{\nu,\tau})\|_{\HT^{p}(\Rn)}.
\end{align*}
For $s<s_{p}$ the right-hand side tends to zero as $\tau\to0$, and it follows that in this case one does not have $\|\lb D\rb^{-s}\F^{-1}(\theta_{\nu,\tau})\|_{\HT^{p}(\Rn)}\lesssim \|\F^{-1}(\theta_{\nu,\tau})\|_{\Hp}$. That is, $\lb D\rb^{-s}:\Hp\to \HT^{p}(\Rn)$ is not bounded.

\eqref{it:sobeq2}: By the Sobolev embeddings for $\Hp$ from \cite[Theorem 7.4]{HaPoRo20}, one has $\|g\|_{\Hp}\lesssim \|\lb D\rb^{s_{p}}g\|_{\HT^{p}(\Rn)}$. For $2<p< \infty$ this also follows from the arguments used above to derive the corresponding inequality for $p\leq 2$. We will use duality to show that $\|\lb D\rb^{s_{p}}g\|_{\HT^{p}(\Rn)}\lesssim \|g\|_{\Hp}$.

First, for $B>1$ write
\[
F_{\nu,\tau,B}:=\{\xi\in\Rn\mid |\xi|\in [\tau^{-1}/B,B\tau^{-1}], |\hat{\xi}-\nu|\leq B\sqrt{\tau}\},
\]
and let $\rho\in C^{\infty}_{c}(\Rn)$ be such that $\rho\equiv 1$ on $F_{\nu,\tau,A}$ and $\rho\equiv0$ outside $F_{\nu,\tau,2A}$. Then $\rho(D)h\in\Sw(\Rn)$ with $\supp(\F(\rho(D)h))\subseteq F_{\nu,\tau,2A}$ for every $h\in\Sw(\Rn)$. Moreover, by taking Fourier transforms in the standard duality pairing $\lb g,h\rb_{\Rn}$ between $g$ and $h$, one obtains
\begin{equation}\label{eq:cutoff}
\lb g,h\rb_{\Rn}=\lb g,\rho(D)h\rb_{\Rn},
\end{equation}
where we used that $\supp(\F g)\subseteq F_{\nu,\tau,A}$, by assumption. Next, by what we have shown in part \eqref{it:sobeq1} with $A$ replaced by $2A$, one has
\begin{equation}\label{eq:inclusion}
\begin{aligned}
I_{1}&:=\{h\in \Sw(\Rn)\mid \supp(\F{h})\subseteq F_{\nu,\tau,2A}, \|\lb D\rb^{-s_{p}}h\|_{\HT^{p'}(\Rn)}\leq c\}\\
\subseteq I_{2}&:=\{h\in\Sw(\Rn)\mid \supp(\F{h})\subseteq F_{\nu,\tau,2A},\|h\|_{\HT^{p'}_{FIO}(\Rn)}\leq 1\}
\end{aligned}
\end{equation}
for some $c>0$ independent of $g$, $\nu$ and $\tau$. Since $\Hp=(\HT^{p'}_{FIO}(\Rn))^{*}$, where the duality pairing is the standard duality pairing between $f_{1}\in\Hp\subseteq\Sw'(\Rn)$ and $f_{2}\in\Sw(\Rn)\subseteq\HT^{p'}_{FIO}(\Rn)$ (see \cite[Proposition 6.8]{HaPoRo20}), and because $\Sw(\Rn)\subseteq \HT^{p'}_{FIO}(\Rn)$ is dense (cf.~\cite[Proposition 6.6]{HaPoRo20}), we can combine \eqref{eq:cutoff} and \eqref{eq:inclusion} to obtain
\begin{align*}
\|g\|_{\Hp}\eqsim \sup_{g\in I_{2}}|\lb g,h\rb_{\Rn}|\geq \sup_{h\in I_{1}}|\lb g,h\rb_{\Rn}|\eqsim \|\lb D\rb^{s_{p}}g\|_{\HT^{p}(\Rn)}.
\end{align*}
This proves \eqref{eq:sobeq2}. For the final statement in \eqref{it:sobeq2} one argues in a similar manner as for $p\leq 2$.
\end{proof}

\begin{remark}\label{rem:examples}
Proposition \ref{prop:examples} shows that the $\Hp$-norm behaves differently depending on whether $p<2$ or $p>2$. Recall from \cite[Theorem 7.4]{HaPoRo20} that
\begin{equation}\label{eq:Sobolev}
W^{s_{p},p}(\Rn)\subseteq\Hp\subseteq W^{-s_{p},p}(\Rn)
\end{equation}
for $1<p<\infty$, with suitable modifications for $p=1$ and $p=\infty$. For $1<p<2$ functions with frequency support in a dyadic-parabolic region have $\Hp$-norm comparable to the $W^{-s_{p},p}(\Rn)$-norm that appears on the right-hand side of \eqref{eq:Sobolev}. Informally speaking, such functions have a ``small" norm on the $L^{p}$-scale. On the other hand, for $2<p<\infty$ the same functions have $\Hp$-norm comparable to the $W^{s_{p},p}(\Rn)$-norm on the left-hand side of \eqref{eq:Sobolev}; here the norm is ``large" on the $L^{p}$-scale.

The fact that the Sobolev embeddings in \eqref{eq:Sobolev} are sharp was already observed in \cite[Remark 7.9]{HaPoRo20}, as a consequence of the optimal $L^{p}$-regularity of FIOs and the fact that $\Hp$ is invariant under suitable FIOs of order zero. On the other hand, Proposition \ref{prop:examples} gives an explicit class of examples that also shows that one of the Sobolev embeddings is optimal.
\end{remark}

\subsubsection*{Acknowledgments}

The authors would like to thank Lixin Yan for helpful discussions, and the anonymous referee for useful remarks.  During the preparation of this manuscript, J.~Rozendaal was supported by grant DP160100941 of the Australian Research Council. L.~Song is supported by  NNSF of China (No.~12071490).

\bibliographystyle{plain}
\bibliography{FLRS}

\end{document}